\documentclass[11pt]{amsart}

\usepackage{amssymb,amsmath,enumerate,epsfig}
\usepackage{amsfonts,color}
\usepackage{a4,latexsym}
\usepackage{parskip}

\usepackage{hyperref}
\hypersetup{pdftitle=Enriques}

\newcommand{\C} {\mathbb{C}}
\newcommand{\Q} {\mathbb{Q}}
\newcommand{\N}  {\mathbb{N}}

\newcommand{\Z}{\mathbb{Z}}

\newcommand{\OO}{\mathcal{O}}

\newcommand{\PP}{\mathbb{P}}
\newcommand{\NS}{\mathop{\rm NS}}
\newcommand{\Num}{\mathop{\rm Num}}
\newcommand{\MW}{\mathop{\rm MW}}
\newcommand{\MWL}{\mathop{\rm MWL}}
\newcommand{\Km}{\mathop{\rm Km}}
\newcommand{\disc}{\mathop{\rm disc}}

\newcommand{\Br}{\mathop{\rm Br}\nolimits}

\newcommand{\Aut}{\mathop{\rm Aut}}
\newcommand{\Triv}{\mathop{\rm Triv}}
\newcommand{\Ess}{\mathop{\rm Ess}}
\newcommand{\Inv}{\mathop{\rm Inv}}
\newcommand{\II}{\mathop{\rm II}}
\newcommand{\jj}{\jmath}

\newtheorem{Theorem}{Theorem}[section]
\newtheorem{Proposition}[Theorem]{Proposition}
\newtheorem{Lemma}[Theorem]{Lemma}

\newtheorem{Corollary}[Theorem]{Corollary}

\theoremstyle{remark}
\newtheorem{Remark}[Theorem]{Remark}

\newtheorem{Example}[Theorem]{Example}

\theoremstyle{definition}

\newtheorem{Definition}[Theorem]{Definition}

\begin{document}

\title{Enriques Surfaces and jacobian elliptic K3 surfaces}

\author{Klaus Hulek}
\address{Institut f\"ur Algebraische Geometrie, Leibniz Universit\"at
  Hannover, Welfengarten 1, 30167 Hannover, Germany} 
\email{hulek@math.uni-hannover.de}

\author{Matthias Sch\"utt}
\email{schuett@math.uni-hannover.de}



\thanks{Partial funding from DFG grant Hu 337/6-1 is gratefully acknowledged}

\dedicatory{Dedicated to Tetsuji Shioda on the occasion of his 70th birthday}

\date{March 18, 2010}

 \begin{abstract}
This paper proposes a new geometric construction of Enriques surfaces. Its starting point are K3 surfaces with jacobian elliptic fibration 
which arise from rational elliptic surfaces by a quadratic base change. 
The Enriques surfaces obtained in this way are characterised by elliptic fibrations with a rational curve as bisection which splits into two sections
on the covering K3 surface.
The construction has applications to the study of Enriques surfaces with specific automorphisms.
It also allows us to answer a question of Beauville about Enriques surfaces whose Brauer groups show an exceptional behaviour.
In a forthcoming paper, we will study arithmetic consequences of our construction.
 \end{abstract}
 
 \maketitle

\section{Introduction}

Enriques surfaces form an important class in the Enriques-Kodaira classification of complex algebraic surfaces.
They arise as quotients of K3 surfaces by a fixed point free involution.
They have been studied both by classical algebro-geometric methods, as well as lattice-theoretic techniques.
It is the main purpose of our paper to combine both approaches. In particular, we want to exhibit a geometric construction which allows us to 
exhibit Enriques involutions concretely and link this approach to lattice-theoretic results.

The main objective of this paper is to propose a new geometric construction of Enriques surfaces starting from K3 surfaces with 
a jacobian elliptic fibration. 
The K3 surfaces arise from rational elliptic surfaces by a quadratic base change and
admit Enriques involutions that can be described explicitly from the geometry.
The Enriques surfaces obtained are characterised as follows:
they possess elliptic fibrations with a rational curve as bisection which splits into two sections
on the covering K3 surface.
To the best of our knowledge, such a situation has so far only been investigated in the case
where the bisection is smooth, for so-called special Enriques surfaces.
For this special case our methods allow us to give an explicit rational description of a $136$-fold cover the moduli space of
special Enriques surfaces (Cor.~\ref{Cor:M=1}).

We point out a few advantages of our approach.
Most importantly, we found the construction very accessible,
even to explicit computations.
We will illustrate this by a number of examples which we hope to be of independent interest.
We will also relate the construction to specific families of Enriques surfaces that have been studied before in the context of special automorphisms.
Finally, we will use our technique to answer a question by Beauville in \cite{Beau} 
by constructing an infinite series of Enriques surfaces whose Brauer groups become trivial under the pull-back to the covering K3 surfaces.

Our original motivation to introduce this new construction was arithmetic in nature.
In a forthcoming paper \cite{HS}, we will employ the techniques developed in this paper
to derive many interesting arithmetic properties of Enriques surfaces, 
in particular about fields of definition and Galois actions on divisors.

A further feature of our construction is that it also works in positive characteristic,
although we will not exploit this property in detail here.

The paper is organised as follows:
Section \ref{s:prel} reviews basics about K3 surfaces and Enriques surfaces
as well as elliptic fibrations and Mordell-Weil lattices. 
The geometric heart of the paper is Section \ref{s:ell}.
It introduces the base change technique to produce interesting K3 surfaces with Enriques involutions
and illustrates the method with several examples.
The connection to automorphisms of Enriques surfaces is investigated in Section \ref{s:cti}.
Specifically we consider Enriques surfaces with cohomologically trivial involution 
after Mukai-Namikawa \cite{MN}
and with finite automorphism group after Kond\=o \cite{K-E}.
The paper concludes with a study of Enriques surfaces whose Brauer groups pull-back identically to zero on the covering K3 surfaces.
We exhibit an infinite number of examples abstractly and one example over $\Q$ explicitly,
thus answering a question by Beauville \cite{Beau}. We would like to point out that a different example of this 
kind was also found independently by Garbagnati and van Geemen \cite{GvG}.






\smallskip

{\fontsize{9pt}{11pt}\selectfont

 \noindent {\bfseries Acknowledgements}: We would like to thank our colleagues
Igor Dolgachev, Bert van Geemen, Shigeyuki Kond\=o, and Masato Kuwata for many discussions.
Special thanks to Valery Gritsenko for useful conversations on lattice embeddings and, in particular, for providing
lattice theoretic arguments used in the proof of Proposition \ref{Prop:M}.
We are grateful to the referee for many helpful comments.
This project was started when the second author held a position at University of Copenhagen.\par}

\section{K3 surfaces and Enriques surfaces}
\label{s:prel}

In this section, we recall the basics about K3 surfaces and Enriques surfaces.
As general references the reader might consult \cite{BHPV} or \cite{CD}.
Throughout this paper we consider smooth projective surfaces,
mainly over $\C$,
but most ideas work over arbitrary fields 
of characteristic other than two.
(In characteristic two the construction breaks down, and one even has to modify the following definition of Enriques surfaces).

\begin{Definition}  

\begin{enumerate}
\item
A smooth projective surface $X$ is called \emph{K3 surface} 
if $X$ is (algebraically) simply connected
with trivial canonical bundle $\omega_X\cong\OO_X$.
\item
An \emph{Enriques surface} $Y$ is a quotient of a K3 surface $X$ by a fixed point free involution $\tau$.
Such an involution is also called \emph{Enriques involution}.
\end{enumerate}
\end{Definition}

Before collecting some invariants of K3 surfaces and Enriques surfaces, let us give a few examples of relevance to this paper.

\begin{Example}
\label{Ex:4,4}
\hspace{3cm}
\begin{enumerate}
\item
\label{it:3}
Let $C$ denote a curve in $\PP^1\times \PP^1$ of bidegree $(4,4)$
with at worst isolated ordinary  rational double points as singularities (ADE).
Then the double cover of $\PP^1\times \PP^1$ branched along $C$ gives a K3 surface $X$
after resolving singularities if necessary.
\item
The \emph{Kummer surface} $\Km(A)$ of an abelian surface $A$ is a K3 surface.
\item
\label{it:5}
Assume that the curve $C$ in \eqref{it:3} is fixed under the involution $\sigma=(-1, -1)$ on $\PP^1\times \PP^1$.
If $C$ does not contain any fixed point of $\sigma$,
then $\sigma$ composed with the deck transformation of the double covering induces 
an Enriques involution $\tau$ on the K3 surface $X$.
Hence the quotient $Y=X/\tau$ is an Enriques surface.
\item
\label{it:6}
Every complex algebraic Kummer surface admits an Enriques involution by \cite{Keum}.
\end{enumerate}
\end{Example}

\subsection{K3 surfaces}
For a K3 surface $X$, we have $b_1(X)=q(X)=0$ and $h^{2,0}(X)=p_g(X)=1$.
Hence $\chi(\OO_X)=2$ and by Noether's formula the topological Euler-Poincar\'e characteristic $e(X)=24$.
Over $\C$, we deduce the following Hodge diamond by Poincar\'e duality:
$$
\begin{array}{ccccc}
&&1&&\\
& 0&& 0&\\
1&&20&&1\\
& 0 && 0 &\\
&&1&&
\end{array}
$$
The curves on a K3 surface $X$ generate 
the \emph{N\'eron-Severi group} $\NS(X)$ consisting of divisors up to algebraic equivalence 
(which agrees with numerical equivalence on K3 surfaces).
Its rank $\rho(X)$ is called the \emph{Picard number}.
The intersection pairing (which coincides in the complex case with the cup-product) equips $\NS(X)$ with the structure of an even integral lattice which,
by the Hodge index theorem, has signature $(1,\rho(X)-1)$.

%

On a complex K3 surface the
cup-product endows the cohomology group $H^2(X,\Z)$ with the structure of a lattice which is even, integral, non-degenerate and unimodular of signature $(3,19)$.
By lattice theory
\[
H^2(X,\Z) \cong 3U + 2E_8(-1) =  \Lambda 
\]
where $U$ denotes the hyperbolic plane and 
$E_8$ is the unique positive-definite unimodular even lattice of rank $8$.
The notation $E_8(-1)$ indicates that the form is multiplied by $-1$, i.e. $E_8(-1)$ is negative
definite.
The lattice $\Lambda$ is called the \emph{K3 lattice}.
The  N\'eron-Severi group $\NS(X)$ embeds primitively into $H^2(X,\Z)$ (as a lattice).
Conversely, let $L$ denote a lattice of rank $r\leq 20$ and signature $(1,r-1)$ 
admitting a primitive embedding into the K3 lattice $\Lambda$.
Moduli theory asserts that K3 surfaces with a primitive embedding $L\hookrightarrow H^2(X,\Z)$ containing an ample class
form a moduli space of dimension $20-r$.
This moduli space is globally irreducible 
if and only if the primitive embedding $L\hookrightarrow\Lambda$ is unique up to isometries.

\subsection{Enriques surfaces}
On a complex Enriques surface $Y$ we also have
$b_1(Y)=q(Y)=0$, but the 
fundamental group is not trivial:
\[
\pi_1(Y) = \Z/2\Z.
\]
The covering surface of the  universal covering
\[
\pi: X \to Y
\]
is a K3 surface.
Clearly $e(Y)=\frac 12 e(X)=12$.
By the holomorphic Lefschetz formula (over $\C$), one obtains
that any fixed point free involution $\tau$ on a K3 surface $X$ has to act as $-1$ on the $2$-form.
Thus  the quotient kills the $2$-form: 
\[
h^{2,0}(Y)=p_g(Y)=0.
\]
Hence all of $H^2(Y)$ is algebraic: $\rho(Y)=b_2(Y)=10$.
Algebraic and numerical equivalence of divisors do not coincide on an Enriques surface $Y$.
Indeed, there is two-torsion in $\NS(Y)$ represented by the canonical divisor $K_Y$.
On the other hand,
the quotient $\NS(Y)_f$ of $\NS(Y)$ by its torsion subgroup 
is again an even unimodular lattice, which
is isomorphic to the so-called  \emph{Enriques lattice}:
\[
\Num(Y) = \NS(Y)_f \cong U+E_8(-1).
\]
Via pull-back under the universal covering, this lattice embeds primitively into $\NS(X)$.
Here the intersection form is multiplied by two, as indicated by the brackets:
\begin{eqnarray}\label{eq:U(2)}
U(2) + E_8(-2) \hookrightarrow \NS(X).
\end{eqnarray}
Conversely, we can pursue the following approach over $\C$:
study lattice-polarised K3 surfaces with a primitive embedding \eqref{eq:U(2)}
(containing an ample class).
As explained before, such complex K3 surfaces form a ten-dimensional moduli space
that is in fact irreducible (the embedding of $U(2)+E_8(-2)$ into $\Lambda$ is unique up to isometries because the embedded Enriques lattice is $2$-elementary).
Thanks to the Torelli theorem,
one can show the following for a complex K3 surface $X$ admitting a primitive embedding \eqref{eq:U(2)}:
$X$ admits an Enriques involution
if and only if there are no $(-2)$-curves in the (negative-definite) complement of the Enriques lattice inside $\NS(X)$:
\[
D\in (U(2)+E_8(-2))^\bot\subset\NS(X) \Rightarrow D^2<-2.
\]
In consequence, one can study moduli of Enriques surfaces through the moduli space 
of lattice-polarised K3 surfaces
after removing the discriminant locus corresponding to those K3 surfaces 
where the above condition is not met.

So far,
Enriques surfaces have mainly been  studied  from the lattice theoretic viewpoint
sketched above.
This paper proposes a new geometric construction 
which allows us to make many properties explicit and to work with precise models.
As a  further advantage, the construction also works in odd positive characteristic.
One of the key ingredients are elliptic fibrations.

\subsection{Elliptic fibrations}
\label{ss:ell}

It is a special feature of K3 surfaces and Enriques surfaces that a single surface may admit several distinct elliptic fibrations.
Here an elliptic fibration is a surjective morphism onto $\PP^1$ such that the general fibre is an elliptic curve.
Necessarily the singular fibres are finite in number.
Their Euler-Poincar\'e characterictics add up to the Euler-Poincar\'e characteristic of the surface.
The possible fibre types have been classified in the complex case by Kodaira  \cite{K} according to the local monodromy.
The classification of singular fibres in the general case is due to Tate \cite{Tate}.

To exhibit an elliptic fibration on a K3 surface,
it suffices to find a divisor $D\neq 0$ of self-intersection $D^2=0$.
Then $D$ or $-D$ is effective by Riemann-Roch.
After subtracting the base locus, the linear system of the resulting effective divisor
gives an elliptic fibration (see \cite{PSS}).
Moreover if the divisor $D$ has the Kodaira type of a singular fibre,
then it necessarily appears as a singular fibre of the given elliptic fibration.

Specifically, one often asks for elliptic fibrations $f: S \to C$ with section, 
i.e.~ with a morphism $O$ from the base curve $C$ to the surface $S$
such that
\[
O: C \to S,\;\;\; f\circ O = \mbox{id}_C.
\]
Such fibrations are also called \emph{jacobian}.
They are particularly nice to work with for several reasons.
For instance one has a Weierstrass form 
and there is a canonical way to compute the Picard number
because the sections themselves form a group, the so-called Mordell-Weil group $\MW(S)$.
This group can be analysed by endowing it (up to torsion) with a lattice structure.
This leads to the notion of Mordell-Weil lattices that we will briefly review in \ref{ss:MWL}.


We now investigate the cases of elliptic K3 surfaces and Enriques surfaces.
Let $X$ be a K3 surface with an elliptic fibration specified by an effective divisor $D$ as above.
Then a section is the same as an irreducible curve $C\subset X$ with $C\cdot D=1$.
Since $C^2$ is even for any curve on a K3 surface by the adjunction formula,
the existence of a section implies that the hyperbolic plane $U$ embeds into $\NS(X)$.
In particular $\rho(X)\geq 2$,
and often one starts by looking for a summand $U$ of $\NS(X)$
in order to exhibit an elliptic fibration with section
(see \ref{ss:MWL} for encoded information about singular fibres and sections).
Note that a general algebraic K3 surface does not have an elliptic fibration (since $\rho= 1$).

In contrast, the Enriques lattice does always contain the hyperbolic plane;
in particular, there is a divisor $D\in\NS(Y)$ with $D^2=0$.
It follows that either $\pm D$ or $\pm2D$ induces an elliptic fibration on $Y$.
Here the factor of two comes into play since every elliptic fibration on an Enriques surface 
has exactly two fibres of multiplicity two.
The canonical divisor can be represented as the difference of the supports of the multiple fibres:
\[
F_1=2G_1,\;\; F_2=2G_2 \; \Longrightarrow \; K_Y=G_1-G_2.
\]
In consequence, any curve on $Y$ meets the multiple fibres with even intersection number and, in particular,
there cannot be a section.


The universal cover of an Enriques surface inherits the elliptic fibration structure.
A general such K3 surface $X$ has $\NS(X)=U(2)+E_8(-2)$;
such a K3 surface does not admit an elliptic fibration with section since
all intersection numbers in $\NS(X)$ are even.
In consequence, elliptic K3 surfaces with section and Enriques involution lie in $9$-dimensional subspaces of the moduli space of K3 surfaces with Enriques involutions.
Nonetheless these surfaces have many interesting applications as we shall see in this paper.

For the sake of convenience, we will allow ourselves some ambiguity of notation.
Since K3 surfaces and Enriques surfaces may admit distinct elliptic fibrations,
it is strictly speaking necessary to specify the fibration 
when we want to discuss its properties such as the Mordell-Weil group.
Whenever the fixed fibration is clear,
we will suppress this subtlety in our notation and simply refer to the elliptic surface $X$,  $\MW(X)$ etc.

\subsection{Mordell-Weil lattices}
\label{ss:MWL}

In \cite{ShMW} Shioda introduced a technique how to 
equip the Mordell-Weil group of an elliptic surface (up to torsion)
with the structure of a lattice (not necessarily integral).
Here only those elliptic surfaces are considered which are not of product type.
For later use, we review the construction in this section.
Details can be found in \cite{ShMW}.

On an elliptic surface $S$ with section algebraic and numerical equivalence coincide.
We fix the zero section $O$ and a general fibre $F$.
These divisor classes span a unimodular lattice. 
Depending on the parity of $O^2=-\chi(\OO_S)<0$,
this lattice is odd or even as follows:
\[ 
\langle O,F\rangle =
\begin{cases}
\langle 1\rangle+\langle-1\rangle
&
\text{ if $\chi(\OO_S)$ is odd},\\ 
U 
&
\text{ if $\chi(\OO_S)$ is even}.\\ 
\end{cases}
\]

Independent of the parity, the above lattice defines an orthogonal projection $\psi$ inside $\NS(S)$.
The image of $\NS(S)$ under $\psi$ is called the essential lattice $\Ess(S)$ of the elliptic surface $S$:
\[
\Ess(S):=\psi(\NS(S)) = \langle O,F\rangle^\bot \subset\NS(S).
\]
The essential lattice is negative definite and even.
It contains all root lattices $R_v$ of fibre components avoiding the zero section as orthogonal summands.
These fibre components are referred to as non-identity components,
as opposed to the distinguished identity component of each fibre that intersects the zero section $O$.
The trivial lattice $\Triv(S)$ is generated by all fibre components and the zero section.
It admits a decomposition
\[
\Triv(S) = \langle O,F\rangle + \sum_v R_v.
\]
A section $P\in\MW(S)$ will be identified with the corresponding divisor class in $\NS(S)$. 
The orthogonal projection $\psi$ acts on the section $P$ as
\begin{eqnarray}
\label{eq:varphi}
\psi(P) = P - O - (\chi(\OO_S)+(P\cdot O)) F.
\end{eqnarray}
Here $\psi(P)^2=-2(\chi(\OO_S)+(P\cdot O))$ is sometimes called the naive height (up to sign).
As soon as $\chi(\OO_S)>1$ (for instance for K3 surfaces),
we have $\psi(P)^2\leq -4$ for any section $P\neq O$.
Hence the root lattices $R_v$ can be recovered from the roots of the essential lattice:
\[
\sum_v R_v = \Ess(S)_\text{root}.
\]
In this way, the essential lattice of an elliptic K3 surface encodes all information about reducible singular fibres.
For the Mordell-Weil group there is an isomorphism
\begin{eqnarray}
\label{eq:MW-NS}
\MW(S) \cong \NS(S)/\Triv(S).
\end{eqnarray}
This can be broken down into two statements. The torsion inside $\MW(S)$ can be computed through the primitive closure $\Triv(S)'$ of the trivial lattice:
\[
\MW(S)_\text{tor} \cong \Triv(S)'/\Triv(S).
\]
For the non-torsion, one considers another orthogonal projection $\varphi$,
this time with respect to $\Triv(S)$.
Since the primitive closure of the trivial lattice need not be unimodular,
we have to tensor with $\Q$ in order to set up $\varphi$ properly:
\[
\varphi: \NS(S)_\Q \to (\Triv(S)_\Q)^\bot\subset\NS(S)_\Q.
\]
Then the isomorphism \eqref{eq:MW-NS} is established by
verifying two facts based on the natural embedding $\MW(S)\hookrightarrow \NS(S)$:
\begin{enumerate}[1.]
\item
The restriction of $\varphi$ to $\MW(S)$ defines a group homomorphism with respect to the group structure of the generic fibre of the elliptic surface $S$; i.e.~
for sections $P,Q\in\MW(S)$ we have $\varphi(P\boxplus Q) = \varphi(P)+\varphi(Q)$ 
where $\boxplus, \boxminus$ denote the group operations in $\MW(S)$.
\item
$\varphi$ induces an isomorphism
$\MW(S)/\MW(S)_\text{tor} \cong \varphi(\NS(S))$.
\end{enumerate}
The quotient 
$\MW(S)/\MW(S)_\text{tor}$
is endowed with the structure of a positive-definite lattice by the height pairing
\[
\langle P,Q\rangle := -\varphi(P)\cdot\varphi(Q), \;\;\; P,Q\in\MW(S).
\]
The resulting lattice is called \emph{Mordell-Weil lattice} and denoted by $\MWL(S)$.
Note that due to tensoring with $\Q$ in the orthogonal projection,
the Mordell-Weil lattice need not be integral.
In fact, the following equality often requires $\MWL(S)$ to have rational (non-integral) discriminant:
\[
|\disc(\NS(S))| = |\disc(\Triv(S))\cdot \disc(\MWL(S))| / (\#\MW(S)_\text{tor})^2.
\]
However, the Mordell-Weil lattice is not very far from being integral since it is known that
\[
\MWL(S) \hookrightarrow (\Triv(S)_{\NS(S)}^\bot)^\vee.
\]
The height pairing can be expressed purely in terms of local contributions.
First the intersection number with the zero section is taken into account (as in the naive height).
Moreover there are correction terms depending on the non-identity fibre components that are met
(cf.~\cite[(8.16)]{ShMW}).
These correction terms account for the possible failure of $\MWL(S)$ to be integral.
As height of a section $P\in\MW(S)$ one defines
\[
h(P) :=\langle P,P\rangle= 2\chi(\OO_S) + 2 (P\cdot O) - \sum_v \mbox{corr}_v(P).
\]
The height of a section is non-negative and zero exactly on torsion sections.
Obviously, if there are no correction terms then the height is an even positive integer.

\section{Elliptic K3 surfaces with Enriques involution}
\label{s:ell}

This section features a new technique to construct K3 surfaces with a specific Enriques involution.
The approach is very geometric and will later be applied to other settings
such as Enriques surfaces with specific automorphisms (Section \ref{s:cti}) and Brauer groups (Section \ref{s:Beau}).
Consequences on the arithmetic side will be investigated in a forthcoming paper \cite{HS}.

The overall idea is to construct Enriques surfaces with an elliptic fibration with a rational bisection 
that splits into two sections on the K3 cover.
As a subcase of this setting, the so-called \emph{special} Enriques surfaces have been studied extensively before, see \ref{ss:special} .
We start by considering the most general setting.

\subsection{Rational elliptic surfaces}

Let $S$ denote a rational elliptic surface with section.
Over algebraically closed fields, it is known that such elliptic surfaces
can be expressed as cubic pencils.
This yields a model as $\PP^2$ blown up in the base points of the pencil.
In general, these base points yield independent sections.
Selecting one of them as the zero section (denoted by $O$),
the remaining eight generate the Mordell-Weil lattice $\MWL(S)$ of rank $8$ up to index three.
By the Shioda-Tate formula, the $\MW$-rank drops only if there are reducible fibres
(which could correspond to infinitely near base points).
In any case, the orthogonal projection $\psi$ with respect to the lattice generated by $O$ and $F$ gives an isometry
\[
\NS(S) = \langle O,F\rangle + E_8(-1) = \langle 1\rangle + \langle-1\rangle + E_8(-1).
\]
By \ref{ss:MWL}, the summand $E_8(-1)$ is identified with the essential lattice $\Ess(S)$.

\subsection{Base change}

We let now $f: \PP^1\to\PP^1$ be a morphism of degree two.
Denote the ramification points by $t_0, t_\infty$.
Then the pull-back $X$ of $S$ via $f$ defines a K3 surface 
under the following general

\textbf{Assumption:}
The fibres of $S$ at the ramification points are reduced.

In other words, we allow smooth fibres and singular fibres of Kodaira types $I_n (n>0), II, III, IV$.
Only Kodaira types $II^*, III^*, IV^*, I_n^* (n\geq 0)$ are disallowed.

In total, we thus obtain a ten-dimensional family of elliptic K3 surfaces with sections (8 dimensions from the rational elliptic surface and two dimensions from the base change).
Pulling back $\NS(S)$ via $f^*$ we see that $\langle 2\rangle + \langle-2\rangle+E_8(-2)$ embeds into $\NS(X)$.
Thanks to the elliptic surface structure on $X$, this can be extended to the primitive embedding
\[
U+E_8(-2) \hookrightarrow \NS(X).
\]
In particular, we see from the moduli dimensions that a general K3 surface in this family has $\NS=U+E_8(-2)$.
In this case, we have an identification of Mordell-Weil lattices
\begin{eqnarray}
\label{eq:MW(2)}
\MWL(X) = \MWL(S)(2).
\end{eqnarray}
Since twice the Enriques lattice does not embed primitively into $U+E_8(-2)$
a general K3 surface from the ten-dimensional family cannot admit an Enriques involution.
We shall now impose a geometric condition that allows us to construct nine-dimensional families of K3 surfaces with Enriques involution.

\subsection{Quadratic twist}
\label{ss:quad}

In order to exhibit K3 surfaces with Enriques involution within our family,
we investigate the quadratic twist of $S$ with respect to $f$.
This is the geometric heart of our construction.
Geometrically the quadratic twist arises as follows.

Let $\imath$ denote the deck transformation for $f$,
i.e.~$\imath\in\Aut(\PP^1)$ such that $f=f\circ \imath$.
Then $\imath$ induces an automorphism of $X$ that we shall also denote by $\imath$.  
The quotient $X/\imath$ returns exactly the rational elliptic surface $S$.
On $X$ we have another natural involution, 
namely the hyperelliptic involution which acts as $(-1)$ on the fibres.
Clearly $\imath$ and $(-1)$ commute,
and their composition gives a further involution
\[
\jj = \imath\circ(-1)\in\Aut(X).
\]
Here $\jj$ has exactly eight fixed points -- 
the two-torsion points in the fixed fibres of $\imath$ at $t_0, t_\infty$
(interpreted in a suitable sense if a fixed fibre is not smooth).
Moreover $\jj$ leaves the $2$-form invariant 
since both $\imath$ and $(-1)$ invert the sign.
Hence $\jj$ is a Nikulin involution,
and the quotient $X/\jj$
has a resolution $X'$ that is K3.
Here $X$ and $X'$  fit into the following commutative diagram:
$$
\begin{array}{ccccc}
&& X &&\\
	&  \stackrel{\scriptstyle f_\imath}\swarrow &  \downarrow & \stackrel{\scriptstyle \;\; f_{\jj}}\searrow &\\
S && \PP^1 && X'\\
 \downarrow & \stackrel{{\scriptstyle f}}\swarrow &&  \stackrel{{\scriptstyle \;\; f}}\searrow &  \downarrow\\
 \PP^1 && = && \PP^1
\end{array}
$$
By construction, $X'$ comes equipped with an elliptic fibration.
In fact, this fibration is almost the same as that of $S$ 
which is why $X'$ is sometimes called the quadratic twist of $S$ with respect to the base change $f$:
only the singular fibres at the fixed points $t_0, t_\infty$ differ.
For instance, if  $S$ has smooth fibres at these points,
then so has $X$;
the quotient $X/\jj$ attains a double line with four isolated rational double points.
In the resolution $X'$, the singularities are replaced by simple rational components.
This results in singular fibres of type $I_0^*$ (and similar types if the fixed fibres fail to be smooth).

\begin{figure}[ht!]
\setlength{\unitlength}{.35in}
\begin{picture}(5,3)(0,0.2)

\thicklines

\put(0,.5){\line(1,0){5}}

\thinlines

\multiput(1,0)(1,0){4}{\line(0,1){3}}

\end{picture}
\caption{Singular fibre of type $I_0^*$}
\label{Fig:0*}
\end{figure}
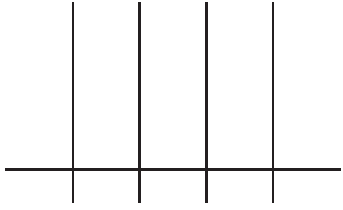

Among the four simple fibre components, the identity component is distinguished 
by intersecting the zero section $O'$ of $X'$.
The non-identity components correspond to the Dynkin diagram $D_4(-1)$.
In general, we thus see that $\NS(X') = U + 2D_4(-1)$.

By definition, $\imath$ and $\jj$ operate exactly oppositely on $H^2(X,\Q)$.
Thus this cohomology group splits into parts coming from $S$ and from $X'$
(in agreement with the fact that $X$ and $X'$ share the same Picard number);
the same applies to the Mordell-Weil group of $X$ up to a $2$-power index, or after tensoring with $\Q$:
\[
\MWL(X)_\Q = f_\imath^*\MWL(S)_\Q + f_{\jj}^*\MWL(X')_\Q.
\]
Henceforth we shall abuse notation and  identify the ramification points $t_0, t_\infty$ and their pre-images among all these elliptic fibrations.
Then we have the following induced action of the above involutions $\imath$ and $\jj$:

\begin{table}[ht!]
\begin{tabular}{ccc}
\hline
sublattice & $\imath^*$ & $\jj^*$\\
\hline
$f_\imath^*\MWL(S)$ & invariant & anti-invariant\\
$f_{\jj}^*\MWL(X')$ & anti-invariant & invariant\\
\hline
\end{tabular}
\end{table}

By translation,
the sections of $S$ and $X'$ give automorphisms of $X$.
Usually they will have infinite order,
but the sections can also be used to define further involutions.
For instance, let $Q\in\MW(S)$ and 
denote the corresponding translation on $X$ by $(\boxplus f_\imath^*Q)\in\Aut(X)$.
Then
\[
\jj\circ(\boxplus f_\imath^*Q)\in\Aut(X)
\]
defines a Nikulin involution.
Two such involutions for $Q, R\in\MW(S)$ are conjugate in $\Aut(X)$
if $Q\boxminus R$ is two-divisible in $\MW(S)$.
In that case, they result in isomorphic quotient surfaces.

\subsection{Enriques involution}
\label{ss:inv}

We now let $P'\in\MW(X')$.
Denote the induced section on $X$ by $P=f_{\jj}^*P'$ and translation by $P$ by $(\boxplus P)$.
Since $P$ is $\imath^*$-anti-invariant,
we derive the following involution on $X$:
\[
\tau = \imath\circ (\boxminus P)\in\Aut(X).
\]
Note that $\tau^*$ acts as $(-1)$ on the $2$-form,
so the quotient $X/\tau$ cannot be K3.
Of course, we would like this quotient to be an Enriques surface.
For this we have to check whether $\tau$ has any fixed points on $X$.
Since $\imath$ exchanges fibres while translation by $P$ fixes them,
we clearly have
\[
\mbox{Fix}(\tau)\subset\mbox{Fix}(\imath) = \{X_{t_0}, X_{t_\infty}\}.
\]
On the fixed fibres, say at $t_0$, the involution $\tau$ acts as
\[
X_{t_0}\ni x \mapsto \tau(x)= x + P_0,\;\;\; P_0 = P(t_0) = P \cap X_{t_0}.
\]
Hence deciding whether $\tau$ acts freely amounts to checking
which points on the fixed fibres $P$ specialises to.
If the fixed fibres are smooth,
there are fixed points (even fixed fibres)
if and only if 
\begin{eqnarray}\label{eq:P-O}
P\cap O\cap  \{X_{t_0}, X_{t_\infty}\} \neq \emptyset.
\end{eqnarray}
The case of non-smooth fibres requires some additional attention,
but the involution can  only  be fixed point free 
if the fixed fibres are semi-stable  (cf.~also \ref{ss:2-t}).
If $\tau$ has fixed points,
then the quotient $X/\tau$ is a rational elliptic surface with one or two multiple fibres (without section).
The condition \eqref{eq:P-O} can be checked on the Nikulin quotient $X'$.
Namely $P'$ intersects one of the four simple components of each ramified fibre on $X'$ (type $I_0^*$),
and $P$ specialises to the two-torsion point on the respective fixed fibre of $X$ that corresponds to this component.
In consequence, $\tau$ has no fixed points 
if and only if $P'$ intersects both fibres of $X'$ at $t_0$ and $t_\infty$
at non-identity components.
An example of this intersection behaviour is sketched in the following figure:

\begin{figure}[ht!]
\setlength{\unitlength}{.45in}
\begin{picture}(5.6,2.8)(0,-0.3)

\thinlines

\qbezier(.9,1.7)(2.4,3)(3.9,1.7)
\put(2.3,1.95){$O'$}

\put(2.3,-.2){$P'$}

\qbezier(.9,.3)(1.3,-.5)(2.4,.2)
\qbezier(2.4,.2)(3.4,1)(3.8,.3)


\thinlines

\put(0.4,1.7){\line(4,1){1}}
\put(0.4,1.4){\line(4,1){1}}
\put(0.4,0.3){\line(4,-1){1}}
\put(0.4,0.6){\line(4,-1){1}}

\put(0,-.2){$X_{t_0}'$}

\thicklines
\put(0.6,2){\line(0,-1){2}}

\thinlines

\put(4.4,1.7){\line(-4,1){1}}
\put(4.4,1.4){\line(-4,1){1}}
\put(4.4,0.3){\line(-4,-1){1}}
\put(4.4,0.6){\line(-4,-1){1}}

\put(4.3,-.2){$X_{t_\infty}'$}

\thicklines
\put(4.2,2){\line(0,-1){2}}

\thinlines

\end{picture}
\caption{Set-up for an Enriques involution $\tau$}
\label{Fig:P-O}
\end{figure}
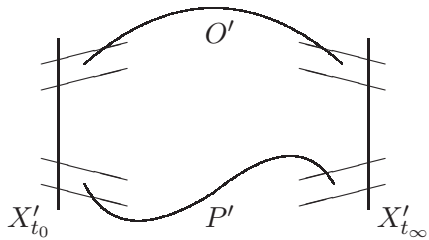

In the sequel, we will refer to $\tau$ as Enriques involution of base change type.

\subsection{Enriques quotient}
\label{ss:special}

In the above set-up, let $Y=X/\tau$ denote the associated Enriques surface.
The given elliptic fibration on $X$ induces an elliptic fibration on $Y$.
Here the smooth fibre of $Y$ at $t$ is isomorphic to the fibres of $X$ at $f^{-1}(t)$ 
-- or to the fibre of the rational elliptic surface $S$ at $t$.
We have seen in \ref{ss:ell}
that no elliptic fibration on $Y$ can have a section.
However, the sections $O$ and $-P$ of the specified elliptic fibration on $X$ are identified under $\tau$ and thus give a bisection for the induced fibration on $Y$.
Our construction is universal in the following sense:

\begin{Proposition}
\label{Prop:split}
Let $Y$ be an Enriques surface with an elliptic fibration with a rational bisection
which splits into sections for the induced fibration on the universal K3 cover $X$.
Then $X$ and $Y$ arise through the base change construction from \ref{ss:inv}.
\end{Proposition}

\begin{proof}
The proof is essentially the same as \cite[Lemma (2.6)(i)]{K-E}.
In detail, let $Y=X/\tau$ and assume that the bisection splits into sections $P,Q$.
Choosing $Q$ as the zero section on $X$, 
we define a new involution $\iota=(\boxminus P)\circ\tau$ on $X$.
By definition, $\iota$ fixes the general fibre $F$ and the zero section $Q$ of the induced elliptic fibration.
Since these fixed divisors generate a hyperbolic plane inside $\NS(X)$, $\iota$ has finite order.
Moreover $\iota^2$ acts trivially on the holomorphic $2$-form.
By the classification of \cite{N0},
the fixed curve $Q$ then implies that $\iota^2=1$.
By construction, the quotient $X/\iota$ is a rational elliptic surface with section.
\end{proof}

There is a subcase of this setting that has been investigated extensively before,
namely \emph{nodal} or \emph{special} Enriques surfaces.
In general, an Enriques surface is called nodal if it contains a nodal curve,
i.e.~a curve of self-intersection $-2$, thus necessarily rational and smooth.
On the K3 cover, such a curve splits into two disjoint smooth rational curves, again with self-intersection $-2$.

It was proved by Cossec in \cite{Cossec}
that the property of being nodal always translates to elliptic fibrations.
Namely for an Enriques surface $Y$, it is equivalent to contain a nodal curve
and to admit a special elliptic fibration, i.e.~an elliptic fibration with a nodal curve as bisection.
For the special case $P\cap O=\emptyset$,
the construction from \ref{ss:inv} thus leads exactly to special Enriques surfaces.
We will investigate two instances of this special case in more detail in \ref{ss:2-t} and \ref{ss:Ex-M=1}.

\subsection{Example: Two-torsion section}
\label{ss:2-t}

The simplest example for an involution of the above kind comes from a two-torsion section.
Namely, if a two-torsion point $P$ is induced from $S$ or equivalently from $X'$, 
then it indeed comes from both surfaces, 
since quadratic twisting preserves exactly the two-torsion in $\MW$.
Note that torsion sections of order relatively prime to the characteristic do not intersect the zero section.
Under our assumptions (characteristic zero or odd), it is thus impossible for $P$ to specialise to $O$ on the fixed fibres.
If both fixed fibres are smooth, we thereby obtain an Enriques involution $\tau$.

If a fixed fibre is not smooth,
$\tau$ can only be fixed point free if the fibre has type $I_n$ for some even $n$ and 
if $P$ meets the opposite component (of the identity component).
Otherwise $\tau$ would have a fixed component because either $P$ would meet the fibre at the identity component, or the fibre type would differ from $I_n, III$ and thus not admit free actions of order two.
As any component of a singular fibre is isomorphic to $\PP^1$, 
there would be at least two fixed points.
(The only remaining fibre type $III$ cannot occur at ramification points unless we are in characteristic two where there is wild ramification.)

The assumption of a two-torsion section cuts the number of moduli to four:
the general such elliptic K3 surface $X$ has eight fibres of type $I_2$ and $I_1$ each;
it arises from a rational elliptic surface $S$ with four singular fibres of each type by a quadratic base change.
Such rational elliptic surfaces have Mordell-Weil rank 4 and four moduli.
Together with the two moduli of the base change, we obtain six moduli for the K3 surface $X$.
In general
$\NS(X)$ comprises a hyperbolic plane, singular fibres and the Mordell-Weil group induced from $S$:
\[
\rho(X) \geq 2 + 8 + 4=14.
\]

\subsection{Lattice theoretic description}
\label{ss:latt}

Consider the sublattice $U\subset \NS(X')$ specified by the zero section $O'$ and a general fibre $F'$ of the elliptic K3 surface $X'$.
Recall the orthogonal projection $\psi$ in $\NS(X')$ with respect to $U$ from \ref{ss:MWL}.
Here 
\[
\psi(P') = P' - O' - (P'\cdot O' + 2) F'
\]
gives the naive height through $\psi(P')^2 = - (4 + 2(P'\cdot O'))$.
The height of $P'$ is obtained 
by subtracting correction terms corresponding to
the reducible singular fibres that $P'$ intersects at non-identity components:
\[
h(P') = -\psi(P')^2 - \sum_v \mbox{contr}_v(P').
\]
As explained, the correction terms depend on the precise fibre component met,
but for fibres of type $I_0^*$, the correction term is always $1$ (or $0$ if the section meets the identity component).

Recall that for $P'$ to induce a fixed point free involution $\tau$ on $X$,
we need that $P'$ intersects the non-identity components of the ramified fibres on $X'$.
This can be expressed lattice theoretically by requiring 
that a certain sublattice $U+L$ embeds primitively into $\NS(X')$.
Here $U+L$ is the trivial lattice $\Triv(X')$ enriched by the section $P'$ meeting the ramified fibres appropriately.
In general, $\Triv(X')=U+2D_4(-1)$,
and the lattice $L$ results from the orthogonal projection $\psi$ with respect to $U$.
Thus 
\[
L=\langle 2D_4(-1), \psi(P')\rangle.
\] 
The condition that $P'$ meets both ramified fibres at non-identity components is encoded in the following diagram.
The diagram incorporates the non-identity components of the two fixed fibres (represented by dots forming two disjoint root lattices of type $D_4$) 
and the divisor $\psi(P')$ (represented by a circle);
as usual, dots or circles are connected if and only if they intersect.

\begin{figure}[ht!]
\setlength{\unitlength}{.45in}
\begin{picture}(4,2.6)(0,-0.2)

\multiput(0,2)(0,-1){3}{\circle*{.1}}
\multiput(4,2)(0,-1){3}{\circle*{.1}}
\multiput(1,1)(2,0){2}{\circle*{.1}}
\put(2,1){\circle{.15}}
\multiput(0,0)(4,0){2}{\line(0,1){2}}
\put(0,1){\line(1,0){4}}

\end{picture}
\caption{Lattice $L$ embedding primitively into $\NS(X')$}
\label{Fig:P'}
\end{figure}
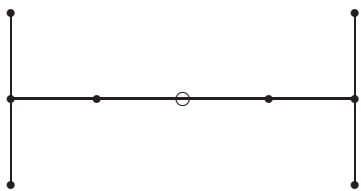

For simplicity, let us write $M=(P'\cdot O')+1\geq 1$.
Due to the correction terms for both fixed fibres in the height of $P'$,
we obtain 
$h(P')=2M$ (unless there are other reducible singular fibres and $P$ meets them at non-identity components).
In any case, the lattice $L$ has discriminant
\[
\disc(L) = - 32 M.
\]
For each $M\geq 1$,
the lattice $U+L$ admits a unique primitive embedding into the K3 lattice by \cite[Thm.~1.14.4]{N}.
The $(U+L)$-polarised K3 surfaces $X'$ thus form an irreducible moduli space of dimension nine.
Quadratic base change ramified at the fixed fibres yields nine-dimensional families
of elliptic K3 surfaces $X$ with an Enriques involution $\tau$ induced by the section $P=f_{\jj}^*P'$.
By \cite{ShMW}, the heights are related by 
\[
h(P)=2h(P').
\]
In general, there are no other reducible fibres, so 
\[
h(P')=2M, \;\;\psi(P')^2=-2M-2.
\]
Hence the general surface $X$ has
\begin{eqnarray}
\label{eq:Ohashi}
\NS(X) = U + E_8(-2) + \langle-4M\rangle.
\end{eqnarray}
Recently Ohashi has classified  K3 surfaces with Picard number $\rho= 11$ and an Enriques involution
\cite{Ohashi}.
The above families form one out of two possible series of K3 surfaces.
Ohashi used these series to prove the following fact: 
though the number of distinct Enriques involution on a K3 surface is always finite,
there can be arbitrarily many. (Here the number of distinct Enriques involutions increases with the number of prime divisors of $M$).
Ohashi's argument is purely lattice-theoretic,
so our geometric construction may be a useful supplement.

\subsection{Example: Special Enriques surfaces} 
\label{ss:Ex-M=1}

We conclude this section by working out explicitly the family for $M=1$, i.e.~special Enriques surfaces.
Rational elliptic surfaces with section have $8$ moduli.
Over fields of characteristic $\neq 2,3$,
a general rational elliptic surface $S$ with section is given by a Weierstrass form
\[
S:\;\; y^2 = x^3 + A(t) x + B(t)
\]
where $A$ and $B$ are polynomials of degree at most $4$ resp.~$6$.
To prevent that the Weierstrass form actually describes a product $E\times\PP^1$, 
we require that $AB$ has at least two different zeroes (possibly including $\infty$);
in particular $A, B$ are not both constant.
The singular fibres of the elliptic surface $S$ are located at the zeroes of the discriminant polynomial
\[
\Delta(S) = 16 (4A^3+27B^2).
\]
By \cite{Tate}, the fibre type can be read off from $\Delta(S)$ and the j-invariant $j(S)=64^3 A^3/\Delta(S)$.
By assumption, there are at least two singular fibres,
and in general there are twelve of them, each irreducible nodal (Kodaira type $I_1$).

Here we can still apply M\"obius transformations to the base curve $\PP^1$ and rescale $x,y$ by squares resp.~cubes to cut down the degrees of freedom from the twelve coefficients of $A$ and $B$ to eight.
In the present situation
it will be convenient to use the M\"obius transformations to normalise the base change as  $f(t)=t^2$
such that it ramifies at $0$ and $\infty$.
Then the quadratic twist $X'$ of $S$ with respect to the base change $f$ is given by
\begin{eqnarray}\label{eq:X'}
X':\;\; y^2 = x^3 + t^2 A(t) x + t^3 B(t).
\end{eqnarray}
This defines a K3 surface if and only if both fibres of $S$ at $0$ and $\infty$ are reduced.
In terms of the polynomials $A, B$, this holds true unless $t^2|A, t^3|B$ or $\deg(A)\leq 2, \deg(B)\leq 3$
(i.e.~outside codimension four-subfamilies).
One has $j(X')=j(S)$, but $\Delta(X')=t^6\Delta(S)$, indicating that only the (singular) fibres at $0$ and $\infty$ differ.

The common quadratic base change of $S$ and $X'$ is the elliptic K3 surface
\[
X:\;\; y^2 = x^3 + A(t^2) x + B(t^2).
\]
Here the deck transformation $\imath$ is given by $t\mapsto -t$, and the hyperelliptic involution $(-1)$ sends $y\mapsto -y$.
Thus the Nikulin involution $\jj$ is defined as
\[
\jj: (x,y,t) \mapsto (x,-y,-t).
\]
In order to induce an Enriques involution on $X$,
we ask for the existence of a section $P'$ on $X'$ as in \ref{ss:inv}.
Here we study the case $M=1$ in the notation of \ref{ss:latt}.
Hence $P'\cdot O'=0$.
This means that $P'$ is given by polynomials
\[
P' = (U'(t), V'(t))\;\;\; \deg(U')\leq 4,\; \deg(V')\leq 6.
\]
We ask that $P'$ meets the non-identity components of the $I_0^*$ fibres at $t=0$ and $\infty$.
By inspection of the Weierstrass form \eqref{eq:X'} this translates into vanishing orders of $U'$ and $V'$ as follows:
\[
t|U', t^2|V' \;\; (\text{at } t=0),\;\;\; \deg(U')\leq 3, \deg(V')\leq 4 \;\; (\text{at } t=\infty).
\]
Thus we can write $P'=(tU, t^2V)$ with polynomials $U,V$ of degree at most two. 
Inserting in \eqref{eq:X'} gives the equation
\begin{eqnarray}
\label{eq:U,V}
tV^2 = U^3 + A U + B.
\end{eqnarray}
This is a polynomial equation of degree six in $t$.
In can be solved generally by choosing $A, U, V$ freely and selecting the resulting $B=tV^2 - (U^3 + A U)$.
One easily checks that this gives a nine-dimensional family of K3 surfaces 
since the polynomials $A, U, V$ have $5+3+3=11$ coefficients in total
compared against the two normalisations left (scale $t$ and independently $x,y$).

Explicitly, let us write $A=a_4t^4+\hdots+a_0, U=u_2t^2+u_1t+u_0$.

\begin{Lemma}
Let $B=tV^2 - (U^3 + A U)$.
Then $X$ admits the section $P=(U(t^2), tV(t^2))$
which induces the Enriques involution $\tau$
unless we are in one of the following two cases:
\begin{enumerate}[(i)]
\item
$a_0=-3u_0^2$ or $a_0=-3u_0^2/4$,
\item 
$a_4=-3u_2^2$ or $a_4=-3u_2^2/4$.
\end{enumerate}
\end{Lemma}

\begin{proof}
It remains to check that $\tau$ has no fixed points on the fixed fibres.
Since we set up $P'$ in such a way that it meets non-identity fibre components,
this is equivalent to the property that the fixed fibres of $X$ at $t=0, \infty$ are smooth.
At $t=0$, the equation for $X$ specialises as
\[
X_0:\;\;\; y^2 = (x-u_0)(x^2+u_0x+a_0+u_0^2).
\]
The fibre is singular if and only if the polynomial on the right hand side has a multiple root.
This happens exactly in the two cases of {\it (i)}.
The analogous argument at $t=\infty$ gives the two cases of {\it (ii)}.
\end{proof}

In the above construction, we have $P\cdot O=P'\cdot O'=0$.
Thus the quotient $Y=X/\tau$ is a special Enriques surface with bisection corresponding to the invariant divisor $O+P$.
With Proposition \ref{Prop:split} it follows that the above construction gives all special Enriques surfaces.
The moduli space involves a subtlety: by  \cite[Prop.~3.6]{Ohashi}, a K3 surface $X$ as above with
$\NS(X) = U + E_8(-2) + \langle -4\rangle$ 
admits a unique Enriques quotient $Y$.
However, the Enriques surface $Y$ generically admits 136 inequivalent special elliptic fibrations by \cite[Thm.~2]{CD-2}.
Hence the moduli space of the above elliptic K3 surfaces gives a 136-fold cover of the moduli space of special Enriques surfaces.
In particular we obtain 

\begin{Corollary}
\label{Cor:M=1}
The moduli space of special Enriques surfaces is unirational.
\end{Corollary}

\begin{Remark}
There are other methods to prove unirationality of the moduli space of special Enriques surfaces. As Igor Dolgachev pointed out
this can also be deduced from the fact that any special Enriques surface is isomorphic to a Reye congruence (see \cite{DR}).
\end{Remark}

\begin{Remark}
If $M>1$, then $P'$ and $O'$ intersect (and likewise $P$ and $O$).
This means that $U', V'$ are no longer polynomials of the given degrees, but involve denominators.
These denominators complicate \eqref{eq:U,V}, so that there is no obvious solution to \eqref{eq:U,V} with polynomial $A, B$ anymore.
\end{Remark}

\section{Geometric relations to other specific Enriques surfaces}
\label{s:cti}

In this section, we shall relate the base change construction from Section \ref{s:ell}
to specific Enriques surfaces.
We will concentrate on geometric questions.
Applications to arithmetic problems will be considered in \cite{HS}.
We will be particularly concerned with automorphisms of Enriques surfaces with a special view towards cohomologically trivial involutions
or finite automorphism groups.

\subsection{Automorphisms}

Often it is possible to compute automorphisms of Enriques surfaces $Y$ on the universal K3 covers $X$.
For this, one looks for automorphisms of $X$ that commute with the Enriques involution.
In this context, the  base change construction from \ref{ss:inv} is particularly convenient.
We continue to employ the notation from Section \ref{s:ell}, confer in particular \ref{ss:quad}, \ref{ss:inv}.
We start by pointing out  one particular feature of the base change construction.
For this we denote by $\Inv(Z)$ the set of involutions of a variety $Z$. 

\begin{Lemma}
\label{Lem:4.1}
In the above setting, there are injections 
\begin{itemize}
\item
$\MW(S)\hookrightarrow\Aut(Y)$, induced by the map $Q\mapsto (\boxplus f_\imath^*Q)\in\Aut(X)$ and 
\item
$\MW(S)\hookrightarrow\Inv(Y)$ induced by the map $Q\mapsto \jj\circ(\boxplus f_\imath^*Q)\in\Inv(X)$.
\end{itemize}
\end{Lemma}

\begin{proof}
We first show that the above maps are well-defined, i.e.~that the given automorphisms on $X$ commute with the Enriques involution $\tau$.
For the translations by $f_\imath^*Q$, this is immediate from $\tau=\imath\circ(\boxminus P)$ since $\MW(X)$ is abelian and $f_\imath^*Q$ is $\imath$-invariant.
Similarly, we find for the Nikulin involution $\jj$:
\begin{eqnarray*}
\jj\circ\tau & = & \imath\circ(-1)\circ\imath\circ(\boxminus P) =  \imath\circ(-1)\circ(\boxplus P)\circ\imath\\
& = & 
  \imath\circ(\boxminus P)\circ(-1)\circ\imath =
\imath\circ(\boxminus P)\circ \imath\circ(-1) = \tau\circ\jj.
\end{eqnarray*}
Trivially the well-definedness now also follows for the Nikulin involutions $\jj\circ(\boxplus f_\imath^*Q)$.
For the injectivity of the maps from Lemma \ref{Lem:4.1}, we first note that clearly $\MW(S)\hookrightarrow\Aut(X)$.
For the induced automorphisms on $Y$ one compares the above automorphisms and their compositions with $\tau$.
Here the action on the fibres of $X$ is totally different:
one automorphism fixes fibres and the other does not.
Hence $\tau\circ(\boxplus f_\iota^*Q)\neq (\boxplus f_\iota^* R)$ for all $Q\neq R\in\MW(S)$
and likewise for the compositions with $\jj$.
\end{proof}

As a corollary, one obtains that for the automorphism group of $Y$ to be finite,
the same has to hold for the  Mordell-Weil group of $S$.
The same argument applies to all elliptic fibrations with section on $X$ that descend to $Y$.
We will take up this reasoning in \ref{ss:finite}.

The action of involutions on $H^2(Y)$ of Enriques surfaces $Y$ has been studied extensively by Mukai and Namikawa.
We are particularly concerned with their results on cohomologically trivial involutions 
that we shall review below (Theorem \ref{Thm:MN}) 
as they relate neatly to our construction.
To illustrate this, we start by considering a series of one-dimensional families of K3 surfaces
which is of great relevance to other issues as well.

\subsection{Specific families}

In this section we consider one-dimensional families of K3 surfaces $X$ with a primitive embedding
\begin{eqnarray}
\label{eq:fam}
U+2E_8(-1)+\langle -2M\rangle \hookrightarrow \NS(X) \;\;\; (M\geq 1).
\end{eqnarray}
For each $M$, these K3 surfaces form an irreducible moduli space,
since the given lattice has a unique embedding into the K3 lattice 
up to isometries by \cite[Thm.~1.14.4]{N}.
These families play a special role in several instances.
In this section we will investigate them for Enriques involutions
and look for the existence of an interesting further involution.
Later two of these families (for $M=2,4$) will turn up 
in the context of Enriques surfaces with finite automorphism groups (\ref{ss:finite}).
In Section \ref{s:Beau}, we will use the families to construct Enriques surfaces
whose Brauer groups show a special behaviour with respect to the pull-back to the universal K3 cover.
Finally, in \cite{HS} these families will turn up in an arithmetic context within the framework of Shioda-Inose structures.

Our next result states when general members of the above families admit Enriques involutions.

\begin{Proposition}
\label{Prop:M}
Let $M\in\N$ and $X$ be a K3 surface with $\NS(X)= U+2E_8(-1)+\langle -2M\rangle$.
\begin{enumerate}[(i)]
\item
If $M$ is odd, then $X$ does not admit an Enriques involution.
\item 
If $M$ is even, then $X$ admits an Enriques involution of base change type as in \ref{ss:inv}.
\end{enumerate}
\end{Proposition}

\begin{Remark}
For special members of the families \eqref{eq:fam} with greater Picard number, 
one has to be more careful.
For instance, if $M$ is even, one has to exclude the possibility of $(-2)$-vectors in the anti-invariant part of $\NS(X)$ for the involution.
One the other hand, in case $M$ is odd, special members of the family \eqref{eq:fam} 
may admit Enriques involutions
(cf.~\cite{HS}, \cite{Sert}).
\end{Remark}

The overall idea behind Proposition \ref{Prop:M} is simple:
embed $E_8(-2)$ canonically into $2E_8(-1)$.
Then $U(2)$ admits a primitive embedding into $U+\langle -2M\rangle$
if and only if $M$ is even.
Below we give a rigorous proof of Proposition \ref{Prop:M} and exhibit the required Enriques involution of base change type in case of even $M$ (\ref{sss:even}).

\subsubsection{Proof of $(i)$}

We will show that the Enriques lattice $U(2)+E_8(-2)$ does not admit a primitive embedding into $\NS(X)$.
To decide about primitivity, we set up some preliminaries.

Let $\pi: S\hookrightarrow L$ be an embedding of integral non-degenerate lattices,
and let $K=S^\bot\subset L$. 
Denote by $S^\vee$ the dual lattice of $S$ and likewise for $L, K$.
Consider the inclusions
\[
S+K \subset L \subset L^\vee\subset S^\vee + K^\vee.
\]
We shall argue with the discriminant groups $D(S)=S^\vee/S$ etc. 
There is an isotropic embedding
\[
H:=L/(S+K) \hookrightarrow D(S) \times D(K)
\]
with projections
\[
\pi_S: H \to D(S),\;\;\; \pi_K: H \to D(K).
\]
We characterise the primitivity of $\pi$ by the following lemma the proof of which is clear.
\begin{Lemma}
\label{Lem:prim}
The embedding $\pi$ is primitive if and only if both $\pi_S$ and $\pi_K$ are injective.
\end{Lemma}

Now we consider an embedding
\begin{eqnarray}
\label{eq:pi-no}
\pi: U(2) + E_8(-2) \hookrightarrow \NS(X)= U+2E_8(-1)+\langle -2M\rangle.
\end{eqnarray}
with $M$ odd. 
We assume that $\pi$ is primitive. 
In order to establish a contradiction, we introduce an auxiliary primitive embedding
\[
i: \NS(X) = U + 2E_8(-1) +\langle -2M\rangle \to 2U+3E_8(-1)  = {\II}_{2,26}
\]
by specifying a primitive element $h$ with $h^2=-2M$ in $U+E_8(-1)$. For instance, we can choose $h=e-Mf\in U$ in the standard basis $e,f$ of $U$.
Note that the orthogonal complement of the image of $i$ is
\[
{\II}_{2,26}\supset N:=\mbox{im}(i)^\bot \cong E_8(-1) + \langle 2M\rangle.
\]
Composing $\pi$ and $i$, we derive an embedding
\[
\bar\pi: U(2)+E_8(-2)+N \hookrightarrow {\II}_{2,26}.
\]
\begin{Lemma}
\label{Lem:bar-pi}
$U(2)+E_8(-2)+N$ does not have a non-trivial overlattice of the same rank in ${\II}_{2,26}$, i.e.~$\bar\pi$ is primitive.
\end{Lemma}

\begin{proof}
Assume there is such an overlattice $\tilde N$.
Then we have an embedding
\[
\bar N:=\tilde N/(U(2)+E_8(-2)+N) \hookrightarrow \underbrace{D(U(2)+E_8(-2))}_{=D(U(2))^5} + D(N).
\]
Since $\pi$ is assumed to be primitive,
$U(2)+E_8(-2)$ is primitive in $\tilde N$.
Hence Lemma \ref{Lem:prim} implies that $\bar N$ embeds into $D(U(2))^5 \cong (\Z/2\Z)^{10}$.
in particular, $\bar N$ must be a $2$-group.
But the only $2$-group in $D(N)=\Z/2M\Z$ is generated by $\frac 12 \bar{h'}$ 
where $h'$ generates $\langle 2N\rangle$.
Hence $\bar N=\Z/2\Z$.
Since $M$ is odd,  we have
\[
\left(\frac 12\bar{h'}\right)^2 \equiv \frac 14 2M \not\equiv 0,1\mod 2\Z.
\]
However, every element $\bar r$ in $D(U(2))\cong(\Z/2\Z)^2$ has
\[
\bar r^2 \equiv 0, 1\mod 2\Z.
\]
This is a contradiction since the embedding $\bar M\hookrightarrow D(U(2)+E_8(-2)) + D(N)$ is isotropic.
\end{proof}

We continue with the primitive embedding $\bar\pi$.
Let $K:=\mbox{im}(\bar\pi)^\bot\subset {\II}_{2,26}$.
Up to sign, it follows that
\[
D(K) = D(U(2)+E_8(-2)) + D(N) \cong (\Z/2\Z)^{10} + \Z/2M\Z.
\]
The minimal number of generators is 11. 
Since $K$ has rank 9, this gives a contradiction to the first assumption that $\pi$ is primitive.
This completes the proof of Proposition \ref{Prop:M} $(i)$. \qed

\begin{Remark}
By \eqref{eq:Ohashi} the Enriques involution of base change type in \ref{ss:inv}
generally leads to K3 surfaces with $\NS=U+E_8(-2)+\langle-4M\rangle$.
It follows from Proposition \ref{Prop:M} $(i)$
that a K3 surface with $\NS=U+E_8(-2)+\langle-2M\rangle$ for odd $M\in\N$
cannot admit an Enriques involution,
since the latter lattice has an obvious primitive embedding into $U+2E_8(-1)+\langle -2M\rangle$.
This was stated without proof in \cite{Ohashi}.
\end{Remark}

\subsubsection{Proof of $(ii)$}
\label{sss:even}

The given decomposition of $\NS(X)$ corresponds to an elliptic fibration on $X$ with section by \ref{ss:ell}.
By \ref{ss:MWL}
this fibration has two singular fibres of type $II^*$ as sole reducible fibres.
Moreover there is a section $P$ of height $2M$ accounting for the 19th generator of $\NS(X)$.

It follows from Tate's algorithm \cite{Tate} 
that elliptic K3 surfaces with the above singular fibres admit the following Weierstrass form after normalisation:
\begin{eqnarray}
\label{eq:Inose}
X:\;\;\; y^2 = x^3 + At^4x+t^5(t^2+Bt+1).
\end{eqnarray}
These elliptic fibrations are related to Shioda-Inose structures (cf.~\cite{Mo}, \cite{Sandwich}).
This connection has great relevance for the arithmetic of Enriques surface as we will exploit in \cite{HS}.
Here we only need that $X$ admits an involution $\imath$ whose quotient is a rational elliptic surface $S$.
Explicitly $\imath$ can be given by
\[
\imath: (x,y,t) \mapsto (x/t^4, y/t^6, 1/t).
\]
In particular, we are in the situation of Section \ref{s:ell}.
Here $S$ has a fibre of type $II^*$ and thus $\MW(S)=\{0\}$.
By \ref{ss:quad}, we deduce $\MW(X)\cong\MW(X')$ for the Nikulin  quotient $X'=X/\jj, \jj=\imath\circ(-1)$.
Thus $P$ is induced by a section $P'\in\MW(X')$.
By \cite[Prop.~8.12]{ShMW}, $P'$ has height $h(P')=h(P)/2=M$.

\textbf{Claim:} $P\cap O\cap\mbox{Fix}(\imath)=\emptyset$.

We shall argue with the Nikulin quotient $X'$.
Recall that $X'$ has $I_0^*$ fibres underneath the (smooth) fixed fibres of $\imath$.
Hence the claim is equivalent to the fact that $P'$ meets both these fibres at non-identity components.
In general, we have 
\[
\NS(X') = U + E_8(-1) + \langle 2D_4(-1), \psi(P')\rangle
\]
where $\psi$ denotes the orthogonal projection in $\NS(X')$ with respect to $U$ as in \ref{ss:MWL}.
From the discriminant group we infer that $P'$ meets at least one $I_0^*$ fibre at a non-identity component.
Otherwise $\NS(X')$ would split into orthogonal summands
\[
\NS(X') = U + E_8(-1) +2D_4(-1) +  \langle  \psi(P')\rangle.
\]
Hence the discriminant group of $\NS(X')$ would give
\[
D(\NS(X')) \supset D(D_4)^2 =  (\Z/2\Z)^4.
\]
Since the lengths of the discriminant groups of $\NS(X')$ and $T(X')$ coincide,
we derive a contradiction from the fact that $T(X')$ has rank three.

Now we only have to distinguish whether $P'$ meets one or both $I_0^*$ fibres at non-identity components.
By \ref{ss:MWL} the height depends on the corresponding correction terms as follows:
\[
h(P') = 4 + 2 (P'\cdot O') - \begin{cases} 1, & \text{one non-id comp},\\
2, & \text{two non-id comp's.}\end{cases}
\]
In the present situation $h(P')=M$ is assumed to be even, 
so $P'$ has to meet both $I_0^*$ fibres at non-identity components.
Hence $P$ meets the fixed fibres of $\imath$ at non-zero two-torsion points, and the claim follows.
By \ref{ss:inv}, we obtain an Enriques involution $\tau=\imath\circ(\boxplus P)$ on $X$.
\qed

\begin{Remark}
For odd $M>1$, the above proof goes through until the last paragraph.
But then the parity of $M$ implies that $P'$ meets a $I_0^*$ fibre at an identity component.
Thus $P$ and $O$ intersect on the corresponding fibre of $X$ and $\tau$ fixes that fibre.
As predicted by Proposition \ref{Prop:M} $(i)$, $\tau$ does not define an Enriques involution on $X$.
\end{Remark}


%

\subsection{Enriques surfaces with cohomologically trivial involution}
\label{ss:6.1}
\label{ss:M}

We start by investigating the families of elliptic K3 surfaces from Proposition \ref{Prop:M} $(ii)$:

\begin{Lemma}
\label{Lem:cti}
Let $X$ be a K3 surface with $\NS(X)=U+2E_8(-1)+\langle -4M\rangle (M\geq 1)$.
In \ref{sss:even}, we constructed an Enriques involution $\tau$ on $X$ of base change type as in \ref{ss:inv}.
The Nikulin involution $\jj$ on $X$ induces an involution $\sigma$ on the Enriques surface $Y=X/\tau$ that acts trivially on $H^2(Y,\Z)$.
\end{Lemma}

\begin{proof}
Consider
the elliptic fibration on $Y$ with a fibre of type $II^*$  induced from the elliptic fibration \eqref{eq:Inose} on $X$.
Recall that fibres of type $II^*$ do not admit any automorphisms and sections always meet the identity component (the only simple component).
Hence $\sigma$ fixes the reducible fibre of $Y$ componentwise,
and it also fixes the bisection $R$ on $Y$ that splits on $X$ into the sections $O$ and $-P$.
Finally let us write the two multiple fibres of the elliptic fibration by $F_1=2G_1$ and $F_2=2G_2$.
Since $\imath$ and $\jj$ induce the same opration on the base curve $\PP^1$,
$\sigma$ fixes the multiple fibres. 

We claim that 
the curves $R, G_1, G_2$ and the components of the $II^*$ fibre generate $H^2(Y,\Z)$.
Since $\sigma$ fixes each of these curves, the claim implies Lemma \ref{Lem:cti}.

To see the claim,
consider the sublattice $L$ of $H^2(Y,\Z)$ generated by $R$ and the components of the $II^*$ fibre.
Since $R$ meets the simple fibre component (twice), $L$ decomposes essentially as in the case of a jacobian elliptic fibration:
\[
L\cong  E_8(-1) + \langle F, R\rangle
=
E_8(-1) + \begin{pmatrix}
0 & 2\\
2 & 2(M-2)
\end{pmatrix}.
\]
In particular, $L$ has discriminant $-4$. 
Now clearly neither $G_1$ nor $G_2$ are contained in $L$,
but $2G_1, 2G_1\in L$.
Hence $L$ has index two in the overlattice $L'=\langle L, G_1\rangle$.
Therefore $L'$ is unimodular.
Since $K_Y=G_1-G_2$ in $H^2(Y,\Z)$,
the given curves generate $H^2(Y,\Z)$ as claimed.
\end{proof}

An involution $\sigma$ on an Enriques surface $Y$ 
which acts trivially on $H^2(Y,\Z)$ is called \emph{cohomologically trivial}.
Note that by the Torelli theorem
cohomologically trivial involutions cannot occur on K3 surfaces.
We have seen examples of such Enriques surfaces in Lemma \ref{Lem:cti}.
The first construction of Enriques surfaces with cohomologically trivial involution was given by Barth and Peters in \cite{BP}.
They exhibited a two-dimensional family $\mathfrak X$ of K3 surfaces
with Enriques involution whose quotients have a cohomologically trivial involution.

On the other hand, Enriques surfaces may also admit numerically trivial involutions
which act trivially on $H^2(Y,\Q)$, but not on $H^2(Y,\Z)$.
The classification of Enriques surfaces with cohomologically or numerically trivial
involutions was initiated by Mukai and Namikawa in \cite{MN}.
After some additions (due to Kond\=o  \cite[Thm.~(1.7)]{K-E}, see \cite{Mukai}) and corrections \cite{Mukai-2}, the classification is as follows:
Every Enriques surface with a numerically trivial involution comes from a Kummer surface of product type in relation with either a Liebermann involution or a Cremona involution.
Enriques surfaces with cohomologically trivial involutions are determined as follows:

\begin{Theorem}
\label{Thm:MN}
Let $Y$ be an Enriques surface with a cohomologically trivial involution.
Then the doubly covering K3 surface belongs to the Barth-Peters family $\mathfrak X$.
\end{Theorem}


We will review the Barth--Peters family in \ref{ss:BP}.
Then we will relate it to our ideas from Section \ref{s:ell} by exhibiting the Enriques involution and the cohomologically trivial involution in terms of the base change construction from \ref{ss:inv}.
%
%
By Theorem \ref{Thm:MN} the one-dimensional families from Lemma \ref{Lem:cti} for each $M>0$ have to be subfamilies of the Barth-Peters family.
For $M=1,2$, we will work out this relation explicitly.

\subsection{Barth-Peters family}
\label{ss:BP}

In \cite{BP}, Barth and Peters exhibited a two-dimensional family of Enriques surfaces with cohomologically trivial involution as follows.
Here we work with the model from \cite{MN}.

Consider two $\PP^1$'s with affine coordinates $x, y$ 
and the four lines $\{x=\pm 1\}$ and $\{y=\pm 1\}$ on the product $\PP^1\times\PP^1$.
Let $E$ be an elliptic curves in $\PP^1\times \PP^1$,
i.e.~$E$ is given by a polynomial $g$ of bidegree $(2,2)$.
Assume that the curve $E$ and the four lines intersect in ordinary rational double points.
Then the double cover of $\PP^1\times\PP^1$ branched along the curve $E$ and the specified lines defines a K3 surface $X$ as its minimal resolution.
Before resolving singularities, we can write affinely
\begin{eqnarray}
\label{eq:BP}
X:\;\; w^2 = (x^2-1)(y^2-1) g(x,y).
\end{eqnarray}
The lines were already set up in such a way that they are invariant under the following involution 
\[
\Aut(\PP^1\times\PP^1)\ni(-1,-1): (x,y) \mapsto (-x,-y).
\]
Assume that the elliptic curve $E$ is also fixed by $(-1,-1)$, but that it does not contain any fixed points.
Then we derive the following Enriques involution $\tau$ on $X$ as in Example \ref{Ex:4,4}.\ref{it:5}:
\[
\Aut(X)\ni\tau: (x,y,w) \mapsto (-x,-y,-w).
\]
As $\tau$ commutes with the deck transformation and with the involution induced by $(-1,-1)$,
the Enriques surface $Y=X/\tau$ inherits an involution $\sigma$.
We now specialise to the case where the elliptic curve $E$
contains the four nodes $(\pm 1, \pm 1)$ where the lines intersect.
The covering K3 surfaces $X$ form a two-dimensional family
that we shall refer to as the Barth-Peters family.
We shall now analyse how this construction fits into our considerations in Section \ref{s:ell}
and describe an alternative way to see that under the given assumption $\sigma$ is cohomologically trivial.

We first study the N\'eron-Severi lattice of the general K3 surface $X$.
By assumption, the double covering gives rational double points of type $D_4$ at the nodes $(\pm 1, \pm 1,0)$ of \eqref{eq:BP}.
The minimal resolution $X$ replaces them by a $D_4$ configuration of $(-2)$-curves.
Together with the pre-images of the lines,
they form the diagram of $(-2)$-curves depicted in Figure \ref{Fig:BP}.

\begin{figure}[ht!]
\setlength{\unitlength}{.45in}
\begin{picture}(5,5)(-0.5,-0.5)

\thinlines

\multiput(0,0)(1,0){5}{\circle*{.1}}
\multiput(0,4)(1,0){5}{\circle*{.1}}
\multiput(0,1)(0,1){3}{\circle*{.1}}
\multiput(4,1)(0,1){3}{\circle*{.1}}
\multiput(2,1)(0,2){2}{\circle*{.1}}
\multiput(1,2)(2,0){2}{\circle*{.1}}

\put(0,0){\line(1,0){4}}
\put(0,4){\line(1,0){4}}
\put(0,0){\line(0,1){4}}
\put(4,0){\line(0,1){4}}

\multiput(0,2)(3,0){2}{\line(1,0){1}}
\multiput(2,0)(0,3){2}{\line(0,1){1}}

\put(-.55,-.4){$(y=-1)$}
\put(3.45,-.4){$(x=-1)$}

\put(-.55,4.25){$(x=1)$}
\put(3.45,4.25){$(y=1)$}

\end{picture}
\caption{$(-2)$-curves on the Barth-Peters family}
\label{Fig:BP}
\end{figure}
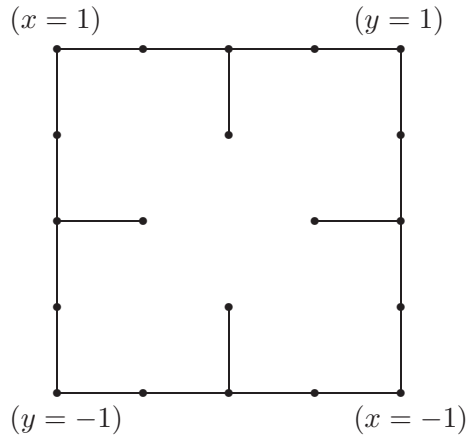

We shall exhibit several jacobian elliptic fibrations on the K3 surface $X$.
For a start, we identify two disjoint divisors of Kodaira type $III^*$ on $X$. 
In the above figure, they can be represented by the $(-2)$-curves above and below the middle horizontal line.
By \ref{ss:ell}, these divisors induce an elliptic fibration on $X$.
The rational curves connecting these two divisors serve as sections, and the remaining two rational curves $C_1, C_2$ represent disjoint fibre components.

Select one of the specified divisors as zero section $O$.
Then the second divisor $P$ gives a two-torsion section in $\MW(X)$.
This can be verified by a case-by-case analysis that compares the possible types of the singular fibres containing the rational curves $C_1, C_2$
 with the resulting height of $P$.

Note that fibres of type $III^*$ do not admit any further torsion sections.
Hence
a general member of the two-dimensional Barth-Peters family satifies $\rho= 18$
with discriminant $-4$:
\begin{eqnarray}
\label{eq:triv}
\Triv(X) = U + 2E_7(-1)+2A_1(-1), \;\;\; \MW(X) = \Z/2\Z.
\end{eqnarray}
Using the discriminant form of $\NS(X)$,
one also verifies that generally $T(X)=U+U(2)$ by the techniques of \cite{N} (cf.~\cite[Prop.~(4.15)]{BP}, mistakenly numbered as (4.5) in \cite{BP}).

In detail, the given elliptic fibration can be obtained from the  rational function
\begin{eqnarray}
\label{eq:t}
\begin{array}{lccc}
t: & X & \dashrightarrow & \PP^1\\
& (x,y,w) & \mapsto & t=\frac{(x+1)(y+1)}{(x-1)(y-1)}.
\end{array}
\end{eqnarray}
Here the involution $(-1,-1)$ on $\PP^1\times\PP^1$ induces $t\mapsto 1/t$.
The latter map can be seen to define an involution $\imath$ on $X$ the quotient of which is a rational elliptic surface $S$ (cf.~\eqref{eq:321} for an explicit equation).
In other words, $X$ is a quadratic base change of $S$:
\[
f: X \stackrel{2:1}{\longrightarrow} S
\]
and we are in the situation of Section \ref{s:ell}.
Here $S$ is uniquely determined by the configuration of singular fibres $[III^*, I_2, I_1]$
with $\MW(S)=\Z/2\Z$, 
since quadratic base changes from the isotrivial rational elliptic surface with configuration $[III^*, III]$
only give a one-dimensional subfamily of the Barth-Peters family.
$S$ has the following model  with affine parameter $s$, reducible fibres of type $III^*$ at $\infty$ and $I_2$ at $s=0$ and two-torsion section $P=(0,0)$:
\begin{eqnarray}\label{eq:321}
S:\;\;\; y^2 = x^3 + x^2 + sx.
\end{eqnarray}
In general, the quadratic base from $S$ 
cannot ramify at a non-smooth fibre since it has to support the two moduli dimensions (or because of the Picard number of $X$).
Hence~the deck transformation $\imath$ has smooth fixed fibres.
Thus $\imath$ composed with translation by the two-torsion section $P$ defines a fixed-point free involution $\tau$ as in \ref{ss:2-t}.
We claim that this involution gives the two-dimensional family of Enriques surfaces $\mathfrak Y=\mathfrak X/\tau$ of Barth-Peters:

\begin{Proposition}
The involution $(-1)$ of the above elliptic fibration induces a cohomologically trivial involution $\sigma$ on the Enriques surface $Y$.
\end{Proposition}

\begin{proof}
The proof resembles that of Lemma \ref{Lem:cti} with a little addition in the spirit of \cite[Prop.~(4.8)]{MN}.
First we note that $(-1)$ commutes with $\tau$ since $P$ is two-torsion.
It follows from the types of the reducible fibres ($I_2, III, III^*$) that $\sigma$ fixes all fibres of $Y$ componentwise.
The same holds for the bisection $R$ that splits on $X$ into $O$ and $P$.
Consider the lattice $L$ generated by $R$ and the components of the reducible fibres.
One computes that $L$ has discriminant $-16$,
so $L$ has index four in $H^2(Y,\Z)_f$.

To $L$, we shall add the 
reduced multiple fibres of two elliptic fibrations on $Y$.
These elliptic fibrations are induced from $X$ as follows:
on the one hand by $t$ as in \eqref{eq:t},
on the other hand after exchanging two signs by 
\[
\begin{array}{lccc}
t': & X & \dashrightarrow & \PP^1\\
& (x,y,w) & \mapsto & t=\frac{(x+1)(y-1)}{(x-1)(y+1)}.
\end{array}
\]
In Figure \ref{Fig:BP}, one can visualise these fibrations as horizontal and vertical.
By adding to $L$ one reduced multiple fibre of each induced fibration on $Y$, 
we obtain a unimodular overlattice of $L$.
Then adding the other reduced multiple fibre gives all of $H^2(Y,\Z)$.
Since in each case $(-1)$ fixes fibres,
we find that
$\sigma$ acts trivially on generators of $H^2(Y,\Z)$.
\end{proof}

\begin{Remark}
The involution $\sigma\in\Inv(Y)$ is induced from $\MW(S)=\Z/2\Z$ as in Lemma \ref{Lem:4.1}.
Namely $(-1)\circ\tau=\jj\circ(\boxplus P)$, so both involutions on $X$ descend to $\sigma\in\Inv(Y)$.
In Figure \ref{Fig:BP}, this can be seen as follows:
$(-1)$ acts as identity, $\tau$ as reflection at the middle point.
Translation by $P$ and $\jj$ act as reflection along the vertical resp.~horizontal axis,
so the composition gives exactly 
$\tau$.
\end{Remark}

\subsection{Specific subfamilies}
We now analyse 
how the families  from Lemma \ref{Lem:cti} fit into the Barth-Peters family.
Abstractly, we can argue with the transcendental lattices.
The families from Lemma \ref{Lem:cti} have generally
\[
T(X) = U + \langle 4M\rangle \;\;\; (M>0).
\]
Each of these transcendental lattices admits a primitive embedding into the transcendental lattice $U+U(2)$ of the general member of the Barth-Peters family.
Lattice theoretically, these primitive embeddings identify the families from Lemma \ref{Lem:cti} as subfamilies of the Barth-Peters family $\mathfrak X$.
In detail, we shall now analyse the families for $M=1$ and $M=2$
and locate them inside the Barth-Peters family  $\mathfrak X$.

\subsection{Family for $M=1$}
\label{ss:M=1}

Consider the one-dimensional family $\mathcal X$ of K3 surfaces with 
\[
U+2E_8(-1)+\langle-4\rangle \hookrightarrow \NS(\mathcal X).
\]
The jacobian elliptic fibrations on a general member of this family 
are classified in \cite{ES-19}:
there are exactly nine different possible configurations of singular fibres.



We single out one fibration whose reducible fibres have type $III^*$ twice and $I_4$ once.
This fibration (generally with $\MW(\mathcal X)=\Z/2\Z$) arises from the above model of elliptic fibrations in the Barth-Peters family  $\mathfrak X$
by merging the two singular fibres of type $I_2$.

Explicitly, we express the family  $\mathfrak X$ of K3 covers from \eqref{eq:BP} in terms of  the family of polynomials
\[
g(x,y) = x^2y^2 + (\lambda-1) x^2 + (\mu-1) y^2 + 1-\lambda-\mu.
\]
For the elliptic fibration with two $III^*$ fibres, we express $y$ in terms of the elliptic parameter $t$ from \eqref{eq:t}.
Elementary transformations lead to the Weierstrass form
\begin{eqnarray}
\label{eq:BPF}
\mathfrak X:\;\;\; w^2 = x(x^2-8(\lambda+\mu-2)t^2x+16(\lambda t+\mu)(\mu t+\lambda)t^3).
\end{eqnarray}
This has two-torsion section $Q=(0,0)$ and reducible singular fibres of type $III^*$ at $t=0, \infty$ and generally $I_2$ at $-\lambda/\mu, -\mu/\lambda$.
Up to a scalar factor, the quadratic base change $f: \PP^1\to\PP^1$ is given by
\[
f: t \mapsto s = \frac{(\lambda t+\mu)(\mu t+\lambda)}{t}.
\]
Obviously $\lambda, \mu$ are independent parameters for these base changes.
This shows that we are indeed dealing with a two-dimensional family  $\mathfrak X$ (see also \ref{sss:BP}).
The merging of the $I_2$ fibres occurs exactly when $\lambda=\pm\mu$,
and both sign choices yield isomorphic families $\mathcal X$.

This family will feature prominently in the arithmetic context of \cite{HS}.
In the present paper, it will return in relation with finite automorphism groups of Enriques surfaces 
in \ref{ss:finite} and in the study of Brauer groups in \ref{ss:-24}.



\subsection{Family for $M=2$}
\label{ss:M=2}

In this section, we are concerned with  the one-dimensional family of K3 surfaces $\mathcal X$ with a primitive embedding
\begin{eqnarray}
\label{eq:NS-8}
U+2E_8(-1)+\langle-8\rangle \hookrightarrow \NS(\mathcal X).
\end{eqnarray}
Our first aim is to exhibit this family explicitly.
For this purpose, we work with the two-dimensional Barth-Peters family $\mathfrak X$ with K3 fibrations with two fibres of type $III^*$ and $I_2$ each and $\MW=\Z/2\Z=\{O,P\}$ as studied in \ref{ss:BP}.
There is only one way to specialise in the family $\mathfrak X$ to surfaces with $\rho\geq 19$ such that generally the above N\'eron-Severi lattice results:
with a section $Q$ disjoint from $O$ that meets exactly one fibre of type $III^*$ and $I_2$ each at non-identity components.
Then the claimed shape of $\NS$ can be checked with the discriminant forms by \cite{N}.

In order to find the section $Q$,
we start with the elliptic fibration from \ref{ss:BP} with two $III^*$ fibres.
Recall that  $\mathfrak X$ sits above the rational elliptic surface $S$ from \eqref{eq:321}.
We take the (family of) quadratic base change $f: \PP^1\to\PP^1$ of the simplest  form possible
\begin{eqnarray}
\label{eq:fab}
f: t\mapsto s=(t-a)(t-b)/t\;\;\; (ab\neq 0).
\end{eqnarray}
Then we can normalise the family $\mathfrak X$ to the integral Weierstrass form
\begin{eqnarray}
\label{eq:8}
\mathfrak X:\;\;\; w^2 = x^3 + t^2 x^2 + t^3(t-a)(t-b) x.
\end{eqnarray}
For the subfamily $\mathcal X$,
write the new section in coordinates $Q=(V,W)$.
We require that $Q\cdot O=0$, so $V, W$ are polynomials in $t$ of degree at most $4$ resp.~$6$.
Now assume that the section  meets the singular fibres at $t=a, \infty$ in non-identity components.
This directly gives the conditions
\[
(t-a)|V,W\;\;\text{ and }\;\; \deg(V)\leq 1,\;\; \deg(W)\leq 3.
\]
Hence $V=\alpha(t-a)$, and we can insert  into \eqref{eq:8} to find a solution directly:
\[
a=-(q^2-1)^2/4, \;\;\; b=2 q^2 (q^2-1),\;\;\; V= q^2 (q^2-1)^2 (4 t+(q^2-1)^2)/4,
\]\[
W=- q(q^2-1) (4 t+(q^2-1)^2) (2 t^2-2 q^2 (q^2-1) t-q^2(q^2-1)^3) /8
\]
Note that the first three expressions are invariant under $q\mapsto -q$, 
so in particular the family can be defined with parameter $\sqrt{q}$,
but the remaining coordinate $W$ of $P$ is negated by conjugation (i.e.~$W/q\in\Q[q^2]$).

Just like the family from \ref{ss:M=1}, this subfamily $\mathcal X$ of the Barth-Peters family $\mathfrak X$ has interesting arithmetic applications that will be exploited in \cite{HS}.
It will also be considered in the context of finite autormorphism groups in the next section.

\subsubsection{Relation between models of the Barth-Peters family}
\label{sss:BP}

We have exhibited two explicit models of the Barth-Peters family  $\mathfrak X$:
the first through the polynomial $g$ in parameters $\lambda, \mu$ in \ref{ss:M=1},
the second in this section through the simple base change \eqref{eq:fab},
Here we describe how these models relate by transforming the second Weierstrass form \eqref{eq:8} to the first  \eqref{eq:BPF}.

For now, we fix some  $\gamma\in\Q(\lambda, \mu)$.
Then the change of variables
\[
(x,w,t) \mapsto \left(\frac{ab}{\gamma} x, \left(\frac{ab}{\gamma}\right)^{3/2} w, (ab)^{1/2} t\right)
\]
transforms the elliptic K3 surfaces \eqref{eq:8} to
\[
w^2 = x^3 + \gamma t^2x^2 + \gamma^2 t^3 (\sqrt bt-\sqrt a)(\sqrt a t - \sqrt b) x.
\]
This Weierstrass form attains the shape of \eqref{eq:BPF} by setting
\[
\gamma=-8(\lambda+\mu-2),\;\;\; \sqrt a = i\lambda/(2(\lambda+\mu-2)),\;\;\; \sqrt b = -i\mu/(2(\lambda+\mu-2)).
\]
In oder to exhibit the subfamily from \ref{ss:M=2}, 
the square roots of $a$ and $b$ can be extracted over $\C(r)$
for $q=(r^2+1/r^2)/2$.
In the next section, we will look into other elliptic fibrations on the given families.
For convenience, we will continue to write them in terms of the parameters $a,b$ resp.~$q$.

\subsection{Enriques surfaces with finite automorphism group}
\label{ss:finite}

By work of Nikulin \cite{N1}, \cite{N2}, a general K3 surfaces with a fixed polarisation has either trivial automorphism group or, if its degree is two, of order two.
In contrast, a general Enriques surface has infinite automorphism group.
In \cite{K-E} Kond\=o classified the Enriques surfaces $Y$ with finite automorphism groups.
Necessarily $Y$ contains a $(-2)$-curve.
By \cite{Cossec}, there is a special elliptic fibration.
On the covering K3 surface $X$, this fibration admits two (disjoint) sections,
and the Enriques involution is of base change type as in \ref{ss:inv} (cf.~\ref{ss:special}).
Associated to $X$ we find a rational elliptic surface $S$.
By Lemma \ref{Lem:4.1}, there is an injection $\MW(S)\hookrightarrow \Aut(Y)$.
Hence for the latter group to be finite, $\MW(S)$ has to be finite, too, i.e.~$S$ has to be extremal.
Moreover, the same argument applies to all special elliptic fibrations of $Y$ simultaneously
(i.e. to all jacobian elliptic fibrations on $X$ that descend to $Y$).
Based on work by Vinberg, Kond\=o classified all possible special elliptic fibrations on $Y$.
These always determine the covering K3 surface:

\begin{Theorem}[Kond\=o]
Let $Y$ be an Enriques surface with finite automorphism group.
Denote by $X$ the universal K3 cover.
Then $X$ belongs to one of the following cases:
the families from \ref{ss:M=1} and \ref{ss:M=2} and five isolated K3 surfaces.
\end{Theorem}

Moreover, the Enriques involution on $X$ is always determined 
by one of the special elliptic fibrations of $Y$.
For instance for the family from \ref{ss:M=1},
Kond\=o proves that $Y$ admits a special elliptic fibration with a fibre of type $III^*$ and an $I_2$ fibre of multiplicity two.
On $X$, this induces exactly the fibration for $\lambda=\pm \mu$ from \ref{ss:M=1}.
In particular, this fibration has $\MW=\Z/2\Z$,
so deck transformation and two-torsion section combine to the only possible Enriques involution of base change type as in \ref{ss:inv}.

%

We now consider those Enriques surfaces with finite automorphism group 
that are covered by the family from \ref{ss:M=2}.
By \cite{K-E} these Enriques surfaces do only admit special elliptic fibrations with the following (combinations of) reducible fibres:
\[
I_4^* \;\;\; \text{ or } \;\;\; I_9 \;\;\; \text{ or } \;\;\; I_1^* \text{ and } I_4 \text{ with multiplicity two}.
\]
The Enriques involution on the family $\mathcal X$ is determined by the last mentioned special elliptic fibration on the Enriques quotient:
the K3 surface inherits an elliptic fibration with two singular fibres of type $I_1^*$ and one fibre of type $I_8$.
For later reference, we denote this elliptic fibration by the pair $(\mathcal X, \pi_1)$
where the morphism $\pi_1: \mathcal X\to \PP^1$ exhibits the specified elliptic structure.
In general, the Mordell-Weil group equals $\Z/4\Z$ by the Shioda-Tate formula.
Sections and reducible fibres are sketched in Figure \ref{Fig:connection}.
Elliptic fibrations with an $N$-torsion point ($N>3$) necessarily arise from the universal elliptic curve for $\Gamma_1(N)$.
For $N=4$, this gives a rational elliptic surface with the same singular fibres as the fixed fibration on the Enriques quotient, given for instance by
\[
y^2 + xy + sy = x^3 + sx^2.
\]
Quadratic base changes ramified at the $I_4$ fibre at $s=0$
give the elliptic fibrations $(\mathcal X, \pi_1)$ 
on the one-dimensional family $\mathcal X$ of K3 surfaces from \ref{ss:M=2}.
In general, $\MW(\mathcal X, \pi_1)=\Z/4\Z$ is invariant under the deck transformation $\imath$, but only the pull-back of the two-torsion section $(-s,0)$ is anti-invariant under $\imath$. 
Hence the Enriques involution $\tau_1$ whose quotient is $\mathcal Y$ is given by the composition of $\imath$ 
with translation by the two-torsion section.
We conclude by relating the above elliptic fibration $(\mathcal X,\pi_1)$ on the one-dimensional family $\mathcal X$ and the model from \ref{ss:M=2}.
%

We start with the full Barth-Peters family  from \ref{ss:BP}.
From the model $\mathfrak X$ in \eqref{eq:8} with two singular fibres of type $III^*$,
it is easy to extract two disjoint divisors of Kodaira type $I_4^*$.
The parameter $u=x/(t(t-a))$ extracts one of them from the $III^*$ fibre at $\infty$ extended by zero section and identity components of the fibres at $t=0, a$.
In the diagram of $(-2)$-curves in Figure \ref{Fig:BP}
the two resulting singular fibres are visible above and below either diagonal,
and the two rational curves on the diagonal form zero and two-torsion section.
One easily finds the following Weierstrass form (a plane cubic in $t,w$):
\begin{eqnarray}
\label{eq:2x4*}
w^2 = t(t^2 + u(u^2+u-b) t -au^4).
\end{eqnarray}
Here the parameter $v=t/u^2$ extracts another elliptic fibration with two singular fibres of type $I_4^*$;
it exchanges the role of the torsion sections and the central double components of the singular fibres.
On the Barth-Peters family, this corresponds to exchanging the moduli $a$ and $b$.
Both families arise through a quadratic base change from the unique rational elliptic surface with $I_4^*$ as sole reducible fibre.
For the above model, this fact can be read off from the involution
\begin{eqnarray}
\label{eq:inv-2x4*}
(t,w,u) \mapsto (b^2 t/u^4, b^3 w/u^6, -b/u).
\end{eqnarray}
With base change and two-torsion section,
each fibration admits an Enriques involution as in \ref{ss:2-t}.
For dimension reasons, this involution does generally not result in an Enriques surface with finite automorphism group
due to the classification in \cite{K-E}.

We now specialise to the family $\mathcal X$ from \ref{ss:M=2} inside the Barth-Peters family $\mathfrak X$ with parameter $q$.
Since the parameters $a$ and $b$ are interchanged, 
the two above fibrations behave differently under  specialisation.
The difference is detected by the additional section of height $2$.
The above fibration \eqref{eq:2x4*}
inherits a section from the elliptic fibration \eqref{eq:8} specialised to the subfamily $\mathcal X$.
Here the section $Q=(V,W)$ of $\mathcal X$ transforms with
$t$-coordinate
\[
T'=q^2(q^2-1)^2.
\]
This section meets the two fibres of type $I_4^*$ as follows: identity component at $u=0$, far component at $u=\infty$. 
Hence the height is indeed two.
Due to the asymmetry in the intersection with the singular fibres, the section cannot come directly from the Nikulin quotient $\mathcal X'$ corresponding to the base change (but this holds for twice the section by \ref{ss:quad}).
In contrast, consider
the elliptic fibration on $\mathcal X$ with $a, b$ interchanged (and specialised to $a(q), b(q)$)
\begin{eqnarray}
\label{eq:2nd4*}
\;\;\;\;(\mathcal X,\pi_2):\;\;\; w^2 = u(u^2 + v(v^2+v+(q^2-1)^2/4) u -2 q^2 (q^2-1)v^4).
\end{eqnarray}
This elliptic surface admits a section $R=(U'',W'')$
that meets both $I_4^*$ fibres symmetrically at the near simple components:
\[
U'' = -(2v+1-q^2)^2v/4.
\]
Indeed the involution corresponding to \eqref{eq:inv-2x4*} inverts $R$ with respect to the group law.
Hence $R$ is induced from the Nikulin quotient $(\mathcal X,\pi_2)'$ associated to the quadratic base change,
and we can construct an Enriques involution $\tau_2$ on $(\mathcal X,\pi_2)$ from the quadratic base change and translation by $-R$ as in \ref{ss:inv}.
We claim that this involution gives Enriques surfaces with finite automorphism group:

\begin{Lemma}
Consider the elliptic fibrations $(\mathcal X,\pi_1), (\mathcal X,\pi_2)$ with Enriques involutions $\tau_1, \tau_2$ of base change type as in \ref{ss:inv}.
On the underlying family of  complex surfaces $\mathcal X$, we  have $\tau_1=\tau_2$.
\end{Lemma}

\begin{proof}
We will argue with the Torelli theorem \cite{PSS}.
Since both involutions act as $(-1)$ on $T(\mathcal X)$, it suffices to study the operation on $\NS(\mathcal X)$.
The fibrations $(\mathcal X,\pi_1), (\mathcal X,\pi_2)$ can be related purely in terms of divisors.
The following figure sketches the configuration of singular fibres and sections on $(\mathcal X,\pi_1)$ and how to extract the reducible fibres and $\MW$-generators of $(\mathcal X,\pi_2)$.

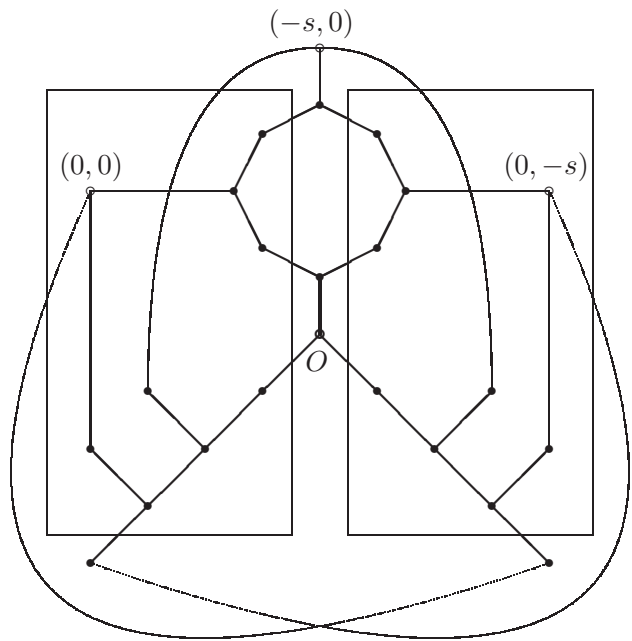
\begin{figure}[ht!]
\setlength{\unitlength}{.30in}
\begin{picture}(6.5,11)(3.8,-5)
\thicklines

%

\put(5.5,3){\circle*{.15}}
\put(6,4){\circle*{.15}}
\put(7,4.5){\circle*{.15}}
\put(8,4){\circle*{.15}}

\put(8.5,3){\circle*{.15}}

\put(5.5,3){\line(1,2){.5}}
\put(6,4){\line(2,1){1}}

\put(8.5,3){\line(-1,2){.5}}
\put(8,4){\line(-2,1){1}}

\put(6,2){\circle*{.15}}
\put(7,1.5){\circle*{.15}}
\put(8,2){\circle*{.15}}

\put(5.5,3){\line(1,-2){.5}}
\put(6,2){\line(2,-1){1}}

\put(8.5,3){\line(-1,-2){.5}}
\put(8,2){\line(-2,-1){1}}

\put(7,1.5){\line(0,-1){1}}
\put(7,.5){\circle{.15}}

\put(6.75,-.15){$O$}

\put(7,.5){\line(1,-1){4}}
\put(7,.5){\line(-1,-1){4}}

\multiput(5,-1.5)(-1,-1){2}{\line(-1,1){1}}
\multiput(4,-.5)(-1,-1){2}{\circle*{.15}}

\multiput(9,-1.5)(1,-1){2}{\line(1,1){1}}
\multiput(10,-.5)(1,-1){2}{\circle*{.15}}

\multiput(8,-.5)(1,-1){4}{\circle*{.15}}
\multiput(6,-.5)(-1,-1){4}{\circle*{.15}}

\thinlines

\put(7,4.5){\line(0,1){1}}
\put(7,5.5){\circle{.15}}

\put(6.1,5.8){$(-s,0)$}

\qbezier(7,5.5)(4,5.5)(4,-.5)
\qbezier(7,5.5)(10,5.5)(10,-.5)

\put(5.5,3){\line(-1,0){2.5}}
\put(3,3){\circle{.15}}

\put(2.45,3.3){$(0,0)$}

\put(3,3){\line(0,-1){4.5}}

\qbezier(3,3)(-2,-8)(11,-3.5)

\put(8.5,3){\line(1,0){2.5}}
\put(11,3){\circle{.15}}

\put(10.2,3.3){$(0,-s)$}

\put(11,3){\line(0,-1){4.5}}

\qbezier(11,3)(16,-8)(3,-3.5)


\put(2.25,-3){\framebox(4.25,7.75){}}
\put(7.5,-3){\framebox(4.25,7.75){}}

\end{picture}
\caption{Two divisors of Kodaira type $I_4^*$ supported on the elliptic fibration $(\mathcal X,\pi_1)$}
\label{Fig:connection}
\end{figure}

Here the old section $O$ gives also a section for the new fibration $(\mathcal X,\pi_2)$ with the two indicated fibres of type $I_4^*$.
Then the opposite fibre component of the original $I_8$ fibre of $(\mathcal X,\pi_1)$ defines a two-torsion section of $(\mathcal X,\pi_2)$,
and the old two-torsion section $(-s,0)$ induces the section $R$ of $(\mathcal X,\pi_2)$ as above.
The identity component of the $I_8$ fibre of $(\mathcal X,\pi_1)$ gives the section $P\boxplus R$.
One easily compares the action of $\tau_1=\imath_1\circ(\boxminus (-s,0))$ and $\tau_2 = \imath_2\circ(\boxminus R)$ on the $(-2)$-curves in Figure \ref{Fig:connection}:
they are identical.
In general, the depicted curves generate $\NS(\mathcal X)$.
Hence $\tau_1=\tau_2$ by Torelli.
\end{proof}

\begin{Remark}
The involutions $\imath\circ(\boxplus R)$ and $\imath\circ(\boxminus R)$ are conjugate by translation by $\pm R$.
Hence they induce isomorphic Enriques quotients, and there is no real ambiguity in the sign of $R$.
\end{Remark}

\section{Brauer groups of Enriques surfaces}
\label{s:Beau}

A complex Enriques surface $Y$ has Brauer group $\Br(Y)=\Z/2\Z$.
It is natural to ask 
what the pull-back of $\Br(Y)$ by the universal covering $\pi: X\to Y$ looks like.
In \cite{Beau},
Beauville showed that 
Enriques surfaces with trivial pull-back
\[
\pi^* \Br(Y) = \{0\}\subset\Br(X).
\]
form an infinite countable union of hypersurfaces in the moduli space of Enriques surfaces.
These hypersurfaces can be characterised in terms of the Enriques involution $\tau$ on $X$ by the following congruence condition:

\begin{Theorem}[Beauville]
\label{Thm:Beau}
In the above notation, 
the following statements are equivalent:
\begin{enumerate}[(i)]
\item
$\pi^* \Br(Y) = \{0\}\subset\Br(X)$;
\item
There is a divisor $D$ on $X$ such that $\tau^*D=-D$ in $\NS(X)$ and $D^2\equiv 2\mod 4$.
\end{enumerate}
\end{Theorem}

The hypersurfaces in the moduli space of Enriques surfaces 
which constitute the surfaces from Theorem \ref{Thm:Beau}
correspond to those K3 surfaces $X$ which admit a primitive embedding
\[
U(2)+E_8(-2) + \langle-2N\rangle \hookrightarrow \NS(X)
\]
for some odd $N>1$ (containing an ample class, but without $(-2)$-vectors in the orthogonal complement in $\NS(X)$).
However, as Beauville remarked, there was no explicit Enriques surface known with the above property.
In particular, it was unclear whether such a surface (K3 $X$ or Enriques $Y$) exists over $\Q$.
In this section we shall solve both problems.
First we will exhibit an infinite number of abstract complex K3 surfaces satisfying condition $(ii)$ of Theorem \ref{Thm:Beau}.
We identify the Enriques involution based on the base change technique from Section \ref{s:ell}.
Then we will work out an explicit example over $\Q$.
Independently an example over $\Q$ was found by Garbagnati and van Geemen  \cite{GvG}.

\subsection{The K3 surfaces}
\label{ss:K3's}

Consider a complex K3 surface $X$ with 
\begin{eqnarray}
\label{eq:M-N}
\NS(X) = U + 2E_8(-1) + \langle -4M\rangle + \langle -2N\rangle
\end{eqnarray}
where $M,N\in\N$ and $N>1$ is odd.
This lattice admits an obvious primitive embedding into the K3 lattice, 
so the existence of some K3 surface $X$ with the specified N\'eron-Severi lattice is clear.
Since $\NS(X)$ has rank $20$,
$X$ is a so-called singular K3 surface and thus defined over some number field (cf.~\cite{SI}).
Note that $\NS(X)$ need not determine $T(X)$ and hence $X$ uniquely.
In fact, it is known that the transcendental lattices of $X$ and its Galois conjugates
cover the whole genus that is determined by $\NS(X)$ through the discriminant form (cf.~\cite{N}).
Here one possible transcendental lattice is
\[
T_{M,N} =  \langle 4M\rangle + \langle 2N\rangle.
\]
For the Enriques involution,
one defines a primitive embedding 
\begin{eqnarray}
\label{eq:Pi}
\Pi:
U(2)+E_8(-2)\hookrightarrow\NS(X)
\end{eqnarray}
summandwise as follows.
Write $e_i (i=1,\hdots,8)$ for a basis of $E_8(-2)$ and $e_{i,1}, e_{i,2}$ for the corresponding basis of the two copies of $E_8(-1)$ inside $\NS(X)$.
Then we set
\[
\Pi|_{E_8(-2)}: e_i \mapsto e_{i,1}+e_{i,2} \;\;\; (i=1,\hdots,8).
\]
On the other hand, we denote the usual basis of $U(2)$ by $f, g$ and likewise $f', g'$ for $U\subset\NS(X)$ 
as well as $h'$ for a generator of the orthogonal summand $\langle -4M\rangle$.
Then we can define $\Pi$ on the summand $U(2)$ by
\[
\Pi|_{U(2)} : (f,g) \mapsto (f',f'+2Mg'+h').
\]
Together, this gives the primitive embedding \eqref{eq:Pi}.
The orthogonal complement of the Enriques lattice embedded by $\Pi$ in $\NS(X)$ can be computed summandwise.
One finds that 
\[
(U(2)^\bot\subset U+\langle-4M\rangle) = \langle 2f'+h'\rangle = \langle-4M\rangle.
\]
In total, we deduce
\[
\NS(X)\supset(U(2)+E_8(-2))^\bot = E_8(-2) +  \langle -4M\rangle + \langle -2N\rangle.
\]
By the assumption $N>1$, the orthogonal complement of the Enriques lattice contains no $(-2)$-vectors.
The primitive embedding $\Pi$ thus determines an Enriques involution $\tau$
on the complex K3 surface $X$.
Here the summand $\langle -2N\rangle$ lies inside the orthogonal complement of the image of $\Pi$,
thus in the $\tau^*$-anti-invariant part.
In other words, the generator $D\in\NS(X)$ of this orthogonal summand satisfies
\[
 D^2=-2N \;\; \text{  and } \;\;  \tau^*D=-D.
\]
Thus condition $(ii)$ of Theorem \ref{Thm:Beau} is satisfied,
and we obtain an Enriques surface $Y=X/\tau$ whose Brauer group pulls back trivially to $X$.

\begin{Remark}
In \cite{HS} we will show that $X$ admits a model over the Hilbert class field $H(-8MN)$
such that the Enriques involution is also defined over the same field.
Thus the same holds for the Enriques surface $Y$.
\end{Remark}

\subsection{The Enriques involution}
\label{ss:ell-M-N}

In this subsection, we will show that the involution $\tau$ has exactly the base change type of Section \ref{s:ell}.
Indeed, the singular K3 surface $X$ arises as a member of the one-dimensional family $\mathcal X$ with
\[
\NS(\mathcal X) = U+ 2E_8(-1)+\langle-4M\rangle
\]
for the given $M$ appearing in Proposition \ref{Prop:M}.
Recall that  $\mathcal X$ is connected to the rational elliptic surface $S$ with singular fibres $II^*, I_1, I_1$ (and $\MW(S)=\{0\}$) through a family of quadratic base changes with deck transformation $\imath$.
Here the base changes are set up in such a way that they endow the pull-back family of $S$
with a section $P$ of height $h(P)=4M$ that does not meet the zero section in the fixed fibres of $\imath$
(but they intersect outside the fixed fibres so that $P\cdot O=2M-2$).
This section is induced from a section $P'$ on the family of Nikulin quotients $\mathcal X'=\mathcal X/\jj$ where $\jj=\imath\circ(-1)$.

On a general member of the family $\mathcal X$, Proposition \ref{Prop:M} tells us that
 $P\cap O\cap\mbox{Fix}(\tau)=\emptyset$.
It remains to check this on the special member $X$.
For $M=1$, this follows from the height $h(P)=4$, since due to the absence of correction terms $P\cdot O=0$.
In general, one can argue with the section $P'$ on $X'$.
The explicit argument will be worked out in \cite{HS}, so we only give a rough idea.
The Shioda-Inose structure predicts the transcendental lattice $T(X')=T(X)(2)$.
Then a similar reasoning as in \ref{sss:even} allows us to deduce the claim from the discriminant form.

Hence $\tau=\imath\circ(\boxminus P)$ defines an Enriques involution as in \ref{ss:inv} on the general member of the family $\mathcal X$ and on the special member $X$.
On the special member $X$ of the family $\mathcal X$ 
there is an independent section $Q$ induced from $X'=X/\jj$.
In the present situation, the height $h(Q)$ equals $2N>2$.
By \ref{ss:MWL}, this implies $Q\cdot O=N-2$.
Since the singular fibres of type $II^*$ have unimodular root lattices, they involve no correction terms.
Hence on $\MW(X)$, the orthogonal projections $\varphi$ with respect to $U$ and $\psi$ with respect to $\Triv(X)$ coincide.
We obtain that
\begin{eqnarray}
\label{eq:P,Q}
\psi(P) = P-O-2MF,\;\;\; \psi(Q) = Q-O-NF,
\end{eqnarray}
and these divisors are mutually orthogonal in $\NS(X)$.
By construction, $\psi(Q)^2=-h(Q)=-2N$.

\begin{Lemma}
\label{Lem:psi-tau}
$\psi(Q)$ is anti-invariant for $\tau^*$: $\tau^*\psi(Q)=-\psi(Q)$ in $\NS(X)$.
\end{Lemma}

\begin{proof}
Abstractly, the anti-invariance is clear since the invariant sublattice sits inside $\Triv(X)+\langle P\rangle$ (see below for the action of $\imath$), and since $\psi(Q)$ is orthogonal to that whole overlattice.
To see the lemma explicitly, we compute the action of $\tau$ on sections.
Recall that we distinguish addition and subtraction of sections with respect to the group law in $\MW(X)$,
denoted by $\boxplus, \boxminus$, and of divisors in $\NS(X)$, denoted as usual by $\pm$.
We use sections in $\MW(X)$ and the corresponding divisors in $\NS(X)$ interchangeably.
\[
\tau(O)= P,\;\;\; \tau(P)=O,\;\;\; \tau(Q) = \imath(Q\boxminus P) = (P\boxminus Q)\in\MW(X).
\]
Hence we see that
\begin{eqnarray}
\label{eq:psi-tau'}
\tau^*\psi(Q) = (P\boxminus Q) - P - NF \;\; \text{ in } \;\NS(X).
\end{eqnarray}
By orthogonality, $h(P\boxminus Q)=h(P)+h(Q)=2M+N$, so 
\begin{eqnarray}
\label{eq:P-Q}
\psi(P\boxminus Q) = (P\boxminus Q) - O - (2M+N)F
\end{eqnarray}
by \ref{ss:MWL}.
Since $\psi$ is a group homomorphism, we can transform \eqref{eq:psi-tau'} with the above relations \eqref{eq:P,Q}, \eqref{eq:P-Q} to obtain:
\begin{eqnarray*}
\label{eq:psi-tau-2}
\tau^*\psi(Q) & = & (\psi(P\boxminus Q) + O + 5F) - (\psi(P) + O + 2F) - 3F\\
& = & \psi(P\boxminus Q)-\psi(P) =  \psi(\boxminus Q) = -\psi(Q).
\end{eqnarray*}
This proves Lemma \ref{Lem:psi-tau} directly.
\end{proof}

\subsection{An explicit singular K3 surface}
\label{ss:-24}

We conclude by producing an explicit example over $\Q$.
We consider the special case 
\[
M=1, N=3.
\]
In this setting, 
$\NS(X)$ has discriminant $-24$, and  the transcendental lattice $T(X)$ has the same discriminant.
By class group theory, there are only two even positive-definite quadratic forms with this discriminant:
$T_{3,1}, T_{1,3}$.
Since they are distinguished by their discriminant forms,
the discriminant form of $\NS(X)$ asserts that
\[
T(X) = T_{1,3} = \langle 4\rangle + \langle 6\rangle.
\]
By the Torelli theorem for singular K3 surfaces \cite{SI},
this determines $X$ uniquely up to isomorphism.
By \cite{S-fields}, $X$ has a model over the Hilbert class field $H(-24)=\Q(\sqrt{2},\sqrt{-3})$.
We will establish a model with the following properties:

\begin{Lemma}
\label{Lem:XX}
The singular K3 surface $X$ admits an elliptic fibration over $\Q$ with two fibres of type $II^*$ 
as in \ref{ss:ell-M-N}
such that the section $P$ is also defined over $\Q$.
Thus the same holds for the Enriques involution $\tau=\imath\circ(\boxminus P)$ and the Enriques surface $Y=X/\tau$.
\end{Lemma}

\begin{Remark}
For the above model of $X$ over $\Q$, 
the section $Q$ can only be defined over $\Q(\sqrt{-3})$.
In the present situation, it is impossible for the sections $P$ and $Q$ to be defined over $\Q$ simultaneously.
This follows from \cite{S-NS}
since 
$\NS(X)$ has to admit a non-trivial Galois action by some quadratic subfield 
of the Hilbert class field $H(-24)=\Q(\sqrt{2},\sqrt{-3})$.
\end{Remark}

We shall exhibit $X$ by locating it inside the family $\mathcal X$ from \ref{ss:M=1}.
The arguments are motivated by the results of the  paper \cite{ES-19},
so we will not give all the details here.
Recall the model related to the Barth-Peters family as an elliptic fibration with two fibres of type $III^*$ and two-torsion section $(0,0)$:
\[
\mathcal X:\;\; y^2 = x^3 + t^2 x^2 + t^3 (t-a)^2 x.
\]
Here we have merged the two other reducible fibres of type $I_2$ in the Barth-Peters family to one fibre of type $I_4$ by choosing $b=a$.
The following figure shows how the parameter $u=x/(t-a)$ extracts two fibres of type $II^*$ from the $III^*$ fibres extended by torsion sections and $I_4$ components:

\begin{figure}[ht!]
\setlength{\unitlength}{.45in}
\begin{picture}(5,5)(-0.5,-0.5)

\thinlines

\multiput(0,0)(1,0){5}{\circle*{.1}}
\multiput(0,4)(1,0){5}{\circle*{.1}}
\multiput(0,1)(0,1){3}{\circle*{.1}}
\multiput(4,1)(0,1){3}{\circle*{.1}}
\multiput(2,.5)(0,1){4}{\circle*{.1}}
\multiput(1.5,2)(1,0){2}{\circle*{.1}}

\multiput(2,1.5)(0,1){2}{\circle{.2}}

\put(0,0){\line(1,0){4}}
\put(0,4){\line(1,0){4}}
\put(0,0){\line(0,1){4}}
\put(4,0){\line(0,1){4}}

\multiput(0,2)(2.5,0){2}{\line(1,0){1.5}}
\multiput(2,0)(0,3.5){2}{\line(0,1){.5}}

\multiput(1.5,2)(.5,.5){2}{\line(1,-1){.5}}
\multiput(1.5,2)(.5,-.5){2}{\line(1,1){.5}}

\put(-.35,2){$O$}
\put(4.1,2){$(0,0)$}



\thinlines
\put(-.5,2.5){\line(1,0){2.25}}
\put(1.75,2.5){\line(0,-1){1.75}}
\put(1.75,.75){\line(1,0){2.5}}
\put(4.25,.75){\line(0,-1){1}}
\put(-.5,2.5){\line(0,-1){2.75}}
\put(-.5,-.25){\line(1,0){4.75}}

\put(2.25,1.5){\line(1,0){2.75}}
\put(2.25,3.25){\line(0,-1){1.75}}
\put(-.25,3.25){\line(1,0){2.5}}
\put(5,4.25){\line(0,-1){2.75}}
\put(-.25,4.25){\line(0,-1){1}}
\put(-.25,4.25){\line(1,0){5.25}}

\end{picture}
\caption{Two divisors of Kodaira type $II^*$ supported on fibre components and torsion sections}
\label{Fig:M=1}
\end{figure}
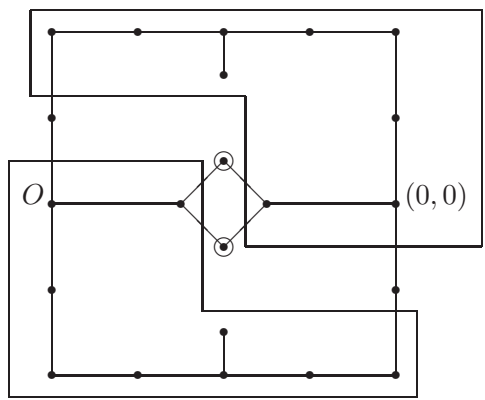

Here the remaining two components of the $I_4$ fibre
(indicated by circles in the figure) form disjoint sections for the new fibration -- zero section and a section $P$ of height $h(P)=4$.
Standard transformations lead to the following Weierstrass form in $t,y$:
\[
\mathcal X:\;\; 
y^2 = t^3+(9a-1)t/9+\left(27\left(u-\frac{a^3}u\right)+81a+2\right)/27.
\]
For the above model, the section $P$ is $\Q(a,u)$-rational with $t$-coordinate 
\[
P_t(u) = (3  u^4+12  u^3 a+6  u^2 a^3+4  u^2 a^2-12  u a^4+3 a^6)/(12 a^2u^2).
\]
One checks that $P\cdot O=0$.
We choose this non-integral model of $\mathcal X$ because the deck transformation $\imath$ is simply given by
\[
\imath: (t,y,u) \mapsto (t,y,-a^3/u).
\]
To determine the specialisations of $\mathcal X$ over $\Q$ with $\rho=20$,
it suffices to investigate the parametrising modular curve.
In the present situation, one can argue with Shioda-Inose structures to see 
that this curve is the Fricke modular curve $X_0(2)/w_+$;
for details, the reader is referred to \cite{ES-19} where also 
all CM-points over $\Q$ are listed.
One of them corresponds to discriminant $d=-20$.

It requires some extra work to exhibit the additional section $Q$ of height $h(Q)=6$ that accounts for the increase of the Picard number.
We know that $Q$ is induced from a section $Q'$ of height $h(Q')=3$ on the Nikulin quotient $X'=X/\jj, \jj=\imath\circ(-1)$.
The  section $Q'$ on $X'$ can be computed explicitly with the techniques of \cite{ES}
which combine a point counting sieve (if necessary) with a $p$-adic Newton iteration.
For the parameter $a$ and the $t$-coordinate $Q_t(u)$ of $Q$, we found
\[
a=-\frac 1{144},\;\;\; Q_t\left(\frac{u}{12^3}\right) = -\dfrac{u^6+222u^5+3039 u^4+36740 u^3+3039 u^2+222 u+1}{2^{18}3^{11}u^2(u-1)^2}.
\]
The $y$-coordinate of $Q$ involves a square root of $-3$,
and conjugation 
inverts $Q$ in $\MW(X)$.
The above representation indicates the invariance of $Q$ under the Nikulin involution $\jj(t,y,u)=(t,-y,12^6/u)$.
Note that $Q$ and $O$ intersect only in the fixed fibre at $u=12^3$,
while $P, Q$ intersect transversally in three points:
a non-zero two-torsion point of the fixed fibre at $u=-12^3$ and 
at the roots of the polynomial $20901888u^2-(255744\pm 31104\sqrt{-3})u+7$ (depending on the sign choices in the $y$-coordinates of $P$ and $Q$).



\end{document}